\documentclass{article}

\usepackage{vmargin}

\setpapersize{A4}
\setmargins{2.5cm}       
{1.5cm}                        
{16.5cm}                      
{23.42cm}                    
{10pt}                           
{1cm}                           
{0pt}                             
{2cm}                           
\date{}

\usepackage{enumerate} 
\usepackage{multicol}
\usepackage{amsmath,amssymb,latexsym,
}
\usepackage{latexsym,amssymb}
\usepackage{amsmath}
\usepackage{mathrsfs}
\usepackage{multicol}
\usepackage{bussproofs}
\usepackage{enumerate}
\usepackage{float}
\usepackage{tikz}
\usepackage{amsfonts}   

\usepackage{amsthm}

\newtheorem{defi}{Definition}[section]
\newtheorem{theo}[defi]{Theorem}
\newtheorem{lem}[defi]{Lemma}

\newtheorem{prop}[defi]{Proposition}

\newtheorem{cor}[defi]{Corollary}

\def\Bt{\mathbb{B}_t}
\def\A{\mathbf{A}}
\def\B{\mathbf{B}}
\def\P{\mathbf{P}}
\def\F{\mathbf{F}}
\def\H{\mathbf{H}}
\def\G{\mathbf{G}}
\def\K{\mathbb{K}}
\def\La{\mathcal{L}}

\def\Fs{Fm_{SK_t}}

\begin{document}

\title{Tense Logic via Truth Degrees: An Algebraic Completeness Result for Kashima’s Calculus}         
\author{Mart\'in Figallo\textsuperscript{1}, Jonathan Sarmiento\textsuperscript{2} and Luis Pezzini\textsuperscript{2}\\[1ex] 
\textsuperscript{1}Departamento de Matem\'atica e Instituto de Matem\'atica (INMABB), \\ Universidad Nacional del Sur, Bah\'ia Blanca, Argentina\\ 
\textsuperscript{2}Departamento de Matem\'atica e Instituto de Ciencias B\'asicas, \\ Universidad Nacional de San Juan, San Juan, Argentina}  
\date{}          
\maketitle

\begin{abstract}
We study the minimal tense logic $K_t$ from an algebraic and
proof-theoretic perspective. We introduce the degree-of-truth-preserving
logic associated with the class of tense Boolean algebras. We then
introduce a sequent calculus for this logic and establish its soundness
and completeness with respect to tense Boolean algebras by means of an
adaptation of the Lindenbaum--Tarski construction. Consequently, this
calculus provides an additional syntactic presentation of the minimal
tense logic $K_t$. We also provide a second, purely syntactic proof of
completeness.

Furthermore, we establish an algebraic soundness and completeness
theorem for Kashima's Gentzen-style calculus. To this end, we develop an
adaptation of the Lindenbaum--Tarski construction to Kashima's nested
sequent framework, which allows us to construct the algebraic semantics
directly from the proof-theoretic system. This yields a direct algebraic
completeness proof for Kashima's calculus and connects its nested-sequent
formulation with the algebraic semantics of tense Boolean algebras.
\end{abstract}



\section{Introduction}

The term {\em tense logic} encompasses a variety of logical systems
and formal frameworks designed to reason about time and temporal
information. More specifically, it refers to the approach introduced
by Arthur Prior, who developed tense logic as a modal framework for
representing temporal notions. His central insight was that temporal
operators can be understood as modal operators whose accessibility
relation is interpreted temporally. This approach has since been
extensively developed in logic and computer science.

Tense logics have a broad range of applications. They provide a formal framework for addressing philosophical questions about time and for analyzing the semantics of temporal expressions in natural language. They also play an important role in artificial intelligence, where they can be used to represent temporal knowledge, and in computer science, where they provide formal tools for specifying, analyzing, and verifying the behavior of programs and systems. Prior's introduction of the well-known tense operators $\P$, $\F$, $\H$, and $\G$ was strongly motivated by the way tense functions in natural language (see \cite{Prior57, Prior67, Prior68}). The intended meanings of these four operators are as follows: $\P$: {\em It has at some time been the case that $\dots$}''; $\F$: {\em It will at some time be the case that $\dots$}''; $\H$: {\em It has always been the case that $\dots$}''; and $\G$: {\em It will always be the case that $\dots$}''.

The {\em minimal tense logic} $K_t$ (see \cite{K}) is the logic whose axioms include all classical tautologies together with the following schemata:
$$\G(\varphi \rightarrow \psi) \rightarrow (\G\varphi \rightarrow \G\psi),$$
$$\H(\varphi \rightarrow \psi) \rightarrow (\H\varphi \rightarrow \H\psi),$$
$$ \varphi \rightarrow \G\P\varphi,$$
$$\varphi \rightarrow \H\F\varphi.$$
The inference rules are classical {\em modus ponens} together with
$$
\frac{\varphi}{\G\varphi},
\qquad
\frac{\varphi}{\H\varphi}.
$$
The first two axiom schemata are the temporal counterparts of the familiar $K$-axiom from modal logic, which explains the notation $K_t$. The third and fourth schemata express the interaction between the past and future operators, while the final two rules are the corresponding necessitation rules.

The system $K_t$ involves two pairs of tense operators: $\P$ and $\H$ for the past, and $\F$ and $\G$ for the future. The operators $\P$ and $\F$ are commonly called the \emph{weak} tense operators, whereas $\H$ and $\G$ are referred to as the \emph{strong} tense operators. These operators are interdefinable by means of the equivalences

$$\P\varphi\equiv \neg\H\neg\varphi,\qquad
\H\varphi\equiv\neg\P\neg\varphi,
\qquad
\F\varphi\equiv \neg\G\neg\varphi,\qquad
\G\varphi\equiv\neg\F\neg\varphi.
$$

From the algebraic point of view, {\em tense Boolean algebras}, also called {\em tense algebras}, provide an algebraic semantics for $K_t$ (see \cite{Burges84,K}). They can be defined as Boolean algebras equipped with two unary normal and join-preserving operators $\F$ and $\P$ that are conjugate, that is,
$$
x \wedge \P y = 0
\quad\text{iff}\quad
\F x \wedge y = 0.
$$
The operators $\H$ and $\G$ can then be defined by duality as $\H x = \neg\P\neg x$ and $\G x = \neg\F\neg x$. Equivalently, tense Boolean algebras can be described by taking $\G$ and $\H$ as primitive operators. These algebras provide the natural algebraic counterpart of the minimal tense logic $K_t$.

On the proof-theoretic side, Kashima \cite{Ka} developed Gentzen-style sequent calculi for a family of propositional tense logics, including $K_t$. A main objective of this work was to obtain cut-free sequent calculi and to establish their soundness and completeness with respect to Kripke semantics. A central feature of Kashima's approach is the use of {\em nested sequents}, which extend ordinary sequents by allowing a hierarchical, tree-like structure. This framework provides a natural proof-theoretic treatment of the interaction between past and future modalities and makes it possible to formulate cut-free calculi for the tense logics under consideration.

The aim of this paper is twofold. First, we introduce the propositional logic that preserves degrees of truth with respect to the class of tense Boolean algebras. We prove that this degree-of-truth-preserving logic coincides with the minimal tense logic $K_t$. Thus, $K_t$ admits, in addition to its usual truth-preserving algebraic semantics, a characterization in terms of preservation of degrees of truth. More precisely, We introduce a sequent calculus for this logic and establish its soundness
and completeness with respect to tense Boolean algebras by means of an
adaptation of the Lindenbaum--Tarski construction. Consequently, this
calculus provides an additional syntactic presentation of the minimal
tense logic $K_t$. We also provide a second, purely syntactic proof of
completeness.

Second, we establish an algebraic soundness and completeness theorem for Kashima's Gentzen-style calculus. The main ingredient of the completeness proof is an adaptation of the Lindenbaum--Tarski construction to Kashima's nested sequent framework. This construction allows us to obtain the required tense Boolean algebra directly from the proof-theoretic system and yields a direct algebraic completeness proof for Kashima's calculus. Thus, our results establish a direct connection between Kashima's nested-sequent calculus and the algebraic semantics of tense Boolean algebras. This connection paves the way for investigating an algebraic version of the cut-elimination theorem in the tense setting, extending the algebraic approach to cut elimination developed in \cite{BJO} to Kashima's framework.

\section{Tense Boolean algebras}

Tense Boolean algebras provide an algebraic semantics for the tense logic $K_t$. They are Boolean algebras equipped with additional unary operators that capture temporal modalities. In this paper, we consider tense Boolean algebras in which the four temporal operators are taken as primitive.

Recall that a Boolean algebra is an algebra
$\A=\langle A,\wedge,\vee,\neg,0,1\rangle$ of type $(2,2,1,0,0)$ such that its reduct
$\langle A,\wedge,\vee,0,1\rangle$ is a bounded distributive lattice and, for every $x\in A$,

$$x\wedge\neg x=0
\qquad\text{and}\qquad
x\vee\neg x=1.
$$
As usual, the lattice order is defined by $x\leq y$ if and only if
$x\wedge y=x$. Moreover, for every $x,y,z\in A$,
$$
x\wedge y\leq z
\quad\text{if and only if}\quad
x\leq\neg y\vee z.
$$

A {\em tense Boolean algebra} is a structure
$\A=\langle A,\wedge,\vee,\neg,\G,\H,\F,\P,0,1\rangle$
of type $(2,2,1,1,1,1,1,0,0)$ such that its reduct
$\langle A,\wedge,\vee,\neg,0,1\rangle$ is a Boolean algebra and the following conditions hold:
\begin{itemize}
\item[(t0)] $\F\neg x=\neg\G x$ and $\P\neg x=\neg\H x$;
\item[(t1)] $\G 1=1$ and $\H 1=1$;
\item[(t2)] $\G(x\wedge y)=\G x\wedge\G y$ and
$\H(x\wedge y)=\H x\wedge\H y$;
\item[(t3)] $x\leq\G\P x$ and $x\leq\H\F x$.
\end{itemize}

The following proposition collects some standard properties of tense Boolean algebras that will be used throughout the paper.

\begin{prop}\label{prop3.1}
Let $\A$ be a tense Boolean algebra. Then, for every $x,y\in A$, the following properties hold:
\begin{enumerate}
\item[\rm (t4)] If $x\leq y$, then
$\G x\leq\G y$, $\H x\leq\H y$,
$\F x\leq\F y$ and $\P x\leq\P y$.

\item[\rm (t5)] $\F 0=0$ and $\P 0=0$.

\item[\rm (t6)] $\F(x\vee y)=\F x\vee\F y$ and
$\P(x\vee y)=\P x\vee\P y$.

\item[\rm (t7)] $\P\G x\leq x$ and $\F\H x\leq x$.

\item[\rm (t8)] $\G(x\vee y)\leq\F x\vee\G y$ and
$\H(x\vee y)\leq\P x\vee\H y$.

\item[\rm (t9)] $\G x\wedge\F y\leq\F(x\wedge y)$ and
$\H x\wedge\P y\leq\P(x\wedge y)$.

\item[\rm (t10)] $x\wedge\F y\leq\F(\P x\wedge y)$ and
$x\wedge\P y\leq\P(\F x\wedge y)$.

\item[\rm (t11)] $\F x\wedge y=0$ if and only if
$x\wedge\P y=0$.

\item[\rm (t12)] $x\vee\G y=1$ if and only if
$\H x\vee y=1$.

\item[\rm (t13)] $\P x\leq y$ if and only if
$x\leq\G y$.

\item[\rm (t14)] $\F x\leq y$ if and only if
$x\leq\H y$.
\end{enumerate}
\end{prop}

We denote by $\Bt$ the class of all tense Boolean algebras.

\section{The degree-of-truth-preserving logic associated with $\Bt$}\label{s3}

In this section, we study the degree-of-truth-preserving logic associated with the class of tense Boolean algebras. We begin by recalling the general notion of a degree-of-truth-preserving consequence relation.

Let $\La$ be a propositional signature containing the connectives
$\wedge$ and $\vee$, and let $\K$ be a class of $\La$-algebras such that,
for every $\A\in\K$, the $\{\wedge,\vee\}$-reduct of $\A$ is a lattice
with greatest element $1$. If $\A,\B\in\K$, we denote by
$\operatorname{Hom}(\A,\B)$ the set of all $\La$-homomorphisms from
$\A$ to $\B$. Let $\mathfrak{Fm}$ be the absolutely free $\La$-algebra generated by
a countably infinite set $Var$ of propositional variables.

Intuitively, a consequence relation preserves degrees of truth when the degree of truth of the premises provides a lower bound for the degree of truth of the conclusion. More precisely, the meet of the truth values of the premises is required to be below the truth value of the conclusion.

\begin{defi}\label{D61}
The logic that preserves degrees of truth with respect to $\K$ is the
propositional logic
\[
\mathbf{L}_{\K}^{\leqslant}
=
\langle Fm,\models_{\K}^{\leqslant}\rangle,
\]
where, for every $\Gamma\subseteq Fm$ and $\alpha\in Fm$,
$\models_{\K}^{\leqslant}$ is defined as follows:
\begin{itemize}

\item[\rm (i)] If $\Gamma$ is finite and non-empty, then
\begin{eqnarray*}
\Gamma\models_{\K}^{\leqslant}\alpha
&\Longleftrightarrow&
\forall \A\in\K,\ \forall h\in \operatorname{Hom}(\mathfrak{Fm},\A),\ \forall a\in A,\\
&&\text{if } h(\gamma)\geq a \text{ for every } \gamma\in\Gamma,
\text{ then } h(\alpha)\geq a.
\end{eqnarray*}

\item[\rm (ii)]

$\emptyset\models_{\K}^{\leqslant}\alpha
\Longleftrightarrow
\forall \A\in\K,\ \forall h\in \operatorname{Hom}(\mathfrak{Fm},\A),\
h(\alpha)=1.$

\item[\rm (iii)] If $\Gamma$ is infinite, then
\[
\Gamma\models_{\K}^{\leqslant}\alpha
\Longleftrightarrow
\text{there exists a finite } \Gamma_0\subseteq\Gamma
\text{ such that }
\Gamma_0\models_{\K}^{\leqslant}\alpha.
\]

\end{itemize}
\end{defi}

\begin{prop}\label{P62}
Let $\alpha\in Fm$ and let $\Gamma\subseteq Fm$ be a finite non-empty set.
Then the following conditions are equivalent:
\begin{itemize}
  \item[\rm (a)] $\Gamma\models_{\K}^{\leqslant}\alpha$;
  \item[\rm (b)] for every $\A\in\K$ and every
  $h\in \operatorname{Hom}(\mathfrak{Fm},\A)$, $\displaystyle\bigwedge\{h(\gamma):\gamma\in\Gamma\}\leq h(\alpha)$.
\end{itemize}
\end{prop}

\begin{proof}
It follows directly from Definition~\ref{D61}.
\end{proof}

\begin{cor}\label{C63}
Let $\alpha_1,\ldots,\alpha_n,\alpha\in Fm$, with $n\geq1$.
Then the following conditions are equivalent:
\begin{itemize}
  \item[\rm (i)] $\alpha_1,\ldots,\alpha_n
  \models_{\K}^{\leqslant}\alpha$;
  \item[\rm (ii)] $\alpha_1\wedge\cdots\wedge\alpha_n
  \models_{\K}^{\leqslant}\alpha$.
\end{itemize}
\end{cor}

\begin{prop}\label{P64}
Let $\alpha,\beta\in Fm$ and $\Gamma,\Delta\subseteq Fm$.
The consequence relation $\models_{\K}^{\leqslant}$ satisfies the following properties:
\begin{itemize}
  \item[\rm (C1)] If $\alpha\in\Gamma$, then
  $\Gamma\models_{\K}^{\leqslant}\alpha$.
  \hfill (Reflexivity)

  \item[\rm (C2)] If $\Delta\models_{\K}^{\leqslant}\alpha$ and
  $\Delta\subseteq\Gamma$, then
  $\Gamma\models_{\K}^{\leqslant}\alpha$.
  \hfill (Monotonicity)

  \item[\rm (C3)] If $\Delta\models_{\K}^{\leqslant}\alpha$ and
  $\Gamma\models_{\K}^{\leqslant}\beta$ for every $\beta\in\Delta$, then
  $\Gamma\models_{\K}^{\leqslant}\alpha$.
  \hfill (Cut for sets)

  \item[\rm (C4)] If $\Delta\models_{\K}^{\leqslant}\alpha$ and
  $\Gamma\cup\{\alpha\}\models_{\K}^{\leqslant}\beta$, then
  $\Delta\cup\Gamma\models_{\K}^{\leqslant}\beta$.
  \hfill (Cut for formulas)

  \item[\rm (C5)] If $\Gamma\models_{\K}^{\leqslant}\alpha$, then
  $\varepsilon(\Gamma)\models_{\K}^{\leqslant}\varepsilon(\alpha)$
  for every endomorphism $\varepsilon$ of $\mathfrak{Fm}$.
  \hfill (Structurality)
\end{itemize}
\end{prop}

\begin{proof}
As an example, we prove (C3). Assume that
$\Delta\models_{\K}^{\leqslant}\alpha$ and
$\Gamma\models_{\K}^{\leqslant}\beta$ for every $\beta\in\Delta$.

We may choose a finite set $\Delta_0\subseteq\Delta$ such that
$\Delta_0\models_{\K}^{\leqslant}\alpha$.
Indeed, if $\Delta$ is finite, we take $\Delta_0=\Delta$; otherwise,
its existence follows from Definition~\ref{D61}.

If $\Delta_0=\emptyset$, then
$\emptyset\models_{\K}^{\leqslant}\alpha$, and it follows directly from
Definition~\ref{D61} that $\Gamma\models_{\K}^{\leqslant}\alpha$.

Suppose that $\Delta_0\neq\emptyset$. For each $\beta\in\Delta_0$, we may
choose a finite set $\Gamma_\beta\subseteq\Gamma$ such that
$\Gamma_\beta\models_{\K}^{\leqslant}\beta$.
Indeed, if $\Gamma$ is finite, we take $\Gamma_\beta=\Gamma$; otherwise,
its existence follows from Definition~\ref{D61}.

Let
\[
\Gamma_0=\bigcup_{\beta\in\Delta_0}\Gamma_\beta.
\]

Then $\Gamma_0$ is finite and, by (C2),
$\Gamma_0\models_{\K}^{\leqslant}\beta$ for every $\beta\in\Delta_0$. If $\Gamma_0=\emptyset$, then
$\emptyset\models_{\K}^{\leqslant}\beta$ for every $\beta\in\Delta_0$.
Hence $h(\beta)=1$ for every $\beta\in\Delta_0$, every $\A\in\K$, and
every $h\in \operatorname{Hom}(\mathfrak{Fm},\A)$. Since
$\Delta_0\models_{\K}^{\leqslant}\alpha$, Proposition~\ref{P62} yields
$h(\alpha)=1$. Thus
$\emptyset\models_{\K}^{\leqslant}\alpha$, and consequently
$\Gamma\models_{\K}^{\leqslant}\alpha$.

Suppose now that $\Gamma_0\neq\emptyset$. By Proposition~\ref{P62},
\[
\bigwedge_{\gamma\in\Gamma_0}h(\gamma)
\leq
\bigwedge_{\beta\in\Delta_0}h(\beta)
\leq
h(\alpha)
\]

for every $\A\in\K$ and every $h\in \operatorname{Hom}(\mathfrak{Fm},\A)$.
Hence
$\Gamma_0\models_{\K}^{\leqslant}\alpha$, and since
$\Gamma_0\subseteq\Gamma$, by (C2) we conclude that
$\Gamma\models_{\K}^{\leqslant}\alpha$.
\end{proof}

\begin{cor}
$\mathbf{L}_{\K}^{\leqslant}$ is a finitary sentential logic; that is,
$\models_{\K}^{\leqslant}$ is a finitary structural consequence relation
on $Fm$.
\end{cor}

In what follows, we consider the logic that preserves degrees of truth
with respect to the class $\Bt$ of tense Boolean algebras. We fix the
propositional signature $\La=\{\wedge,\vee,\neg,\G,\H,\F,\P\}$. For each
$\A\in\Bt$, we denote by $\A_{\La}$ its $\La$-reduct. Thus, valuations are
homomorphisms $h:\mathfrak{Fm}\longrightarrow\A_{\La}$. The
degree-of-truth-preserving logic associated with $\Bt$ is denoted by
\[
\mathbf{L}_{\Bt}^{\leqslant}
=
\langle Fm,\models_{\Bt}^{\leqslant}\rangle.
\]
We extend the consequence relation to finite sets of conclusions as
follows. Let $\Gamma,\Delta\subseteq Fm$ be finite sets. We say that
\emph{$\Delta$ is a consequence of $\Gamma$ in
$\mathbf{L}_{\Bt}^{\leqslant}$}, denoted by
\[
\Gamma\models_{\Bt}^{\leqslant}\Delta,
\]
if, for every $\A\in\Bt$ and every
$h\in\operatorname{Hom}(\mathfrak{Fm},\A_{\La})$,
\[
\bigwedge\{h(\gamma):\gamma\in\Gamma\}
\leq
\bigvee\{h(\delta):\delta\in\Delta\}.
\]
As usual, we adopt the conventions
\[
\bigwedge\emptyset=1
\qquad\text{and}\qquad
\bigvee\emptyset=0.
\]
When the set $\Delta$ of conclusions is a singleton, this extension
recovers the original consequence relation of
$\mathbf{L}_{\Bt}^{\leqslant}$.

\subsection{A Gentzen-style system for $\mathbf{L}_{\Bt}^{\leqslant}$}
We consider a multiple-conclusion sequent calculus with sequents of the form $\Gamma\Rightarrow\Delta,$ where $\Gamma$ and $\Delta$ are finite sets of formulas. For
$\#\in\{\H,\G,\F,\P\}$, we write $\#\Gamma=\{\#\gamma:\gamma\in\Gamma\}.$ Let $\mathfrak{Bt}$ be the sequent calculus given by the following axioms
and rules.

\medskip

\noindent{\bf Axioms}
\[
\alpha\Rightarrow\alpha
\]

\medskip

\noindent{\bf Structural rules}
\[
(we)_l\quad
\frac{\Gamma\Rightarrow\Delta}
     {\Gamma,\alpha\Rightarrow\Delta}
\qquad
(we)_r\quad
\frac{\Gamma\Rightarrow\Delta}
     {\Gamma\Rightarrow\Delta,\alpha}
\qquad
(cut)\quad
\frac{\Gamma\Rightarrow\Delta,\alpha
      \qquad
      \alpha,\Gamma\Rightarrow\Delta}
     {\Gamma\Rightarrow\Delta}
\]

\medskip

\noindent{\bf Logical rules}

\begin{center}
\begin{tabular}{p{0.45\textwidth} p{0.45\textwidth}}

$(\wedge\Rightarrow)\quad
\displaystyle
\frac{\Gamma,\alpha,\beta\Rightarrow\Delta}
     {\Gamma,\alpha\wedge\beta\Rightarrow\Delta}$
&
$(\Rightarrow\wedge)\quad
\displaystyle
\frac{\Gamma\Rightarrow\Delta,\alpha
      \qquad
      \Gamma\Rightarrow\Delta,\beta}
     {\Gamma\Rightarrow\Delta,\alpha\wedge\beta}$
\\[6mm]

$(\vee\Rightarrow)\quad
\displaystyle
\frac{\Gamma,\alpha\Rightarrow\Delta
      \qquad
      \Gamma,\beta\Rightarrow\Delta}
     {\Gamma,\alpha\vee\beta\Rightarrow\Delta}$
&
$(\Rightarrow\vee)\quad
\displaystyle
\frac{\Gamma\Rightarrow\Delta,\alpha,\beta}
     {\Gamma\Rightarrow\Delta,\alpha\vee\beta}$
\\[6mm]

$(\neg\,\Rightarrow)\quad
\displaystyle
\frac{\Gamma\Rightarrow\Delta,\alpha}
     {\Gamma,\neg\alpha\Rightarrow\Delta}$
&
$(\Rightarrow\,\neg)\quad
\displaystyle
\frac{\alpha,\Gamma\Rightarrow\Delta}
     {\Gamma\Rightarrow\neg\alpha,\Delta}$
\\[6mm]

$(\G^{\ast})\quad
\displaystyle
\frac{\Gamma\Rightarrow\Delta,\alpha}
     {\G\Gamma\Rightarrow\F\Delta,\G\alpha}$
&
$(\H^{\ast})\quad
\displaystyle
\frac{\Gamma\Rightarrow\Delta,\alpha}
     {\H\Gamma\Rightarrow\P\Delta,\H\alpha}$
\\[6mm]

$({\ast}\F)\quad
\displaystyle
\frac{\Gamma,\alpha\Rightarrow\Delta}
     {\G\Gamma,\F\alpha\Rightarrow\F\Delta}$
&
$({\ast}\P)\quad
\displaystyle
\frac{\Gamma,\alpha\Rightarrow\Delta}
     {\H\Gamma,\P\alpha\Rightarrow\P\Delta}$
\\[6mm]

$(\P\G)\quad
\displaystyle
\frac{\alpha\Rightarrow\Delta}
     {\P\G\alpha\Rightarrow\Delta}$
&
$(\F\H)\quad
\displaystyle
\frac{\alpha\Rightarrow\Delta}
     {\F\H\alpha\Rightarrow\Delta}$
\\[6mm]

$(\G\P)\quad
\displaystyle
\frac{\Gamma\Rightarrow\alpha}
     {\Gamma\Rightarrow\G\P\alpha}$
&
$(\H\F)\quad
\displaystyle
\frac{\Gamma\Rightarrow\alpha}
     {\Gamma\Rightarrow\H\F\alpha}$

\end{tabular}
\end{center}

A $\mathfrak{Bt}$-proof of a sequent $\Gamma\Rightarrow\Delta$ is a finite
tree whose root is $\Gamma\Rightarrow\Delta$, whose leaves are axioms of
$\mathfrak{Bt}$, and whose internal nodes are obtained by applications
of the rules of $\mathfrak{Bt}$. We write
$\mathfrak{Bt}\vdash\Gamma\Rightarrow\Delta$ if there exists a
$\mathfrak{Bt}$-proof of $\Gamma\Rightarrow\Delta$.
In this case, we say that $\Gamma\Rightarrow\Delta$ is provable in
$\mathfrak{Bt}$. We say that a formula $\alpha$ is provable in
$\mathfrak{Bt}$ if the sequent $\Rightarrow\alpha$ is provable in
$\mathfrak{Bt}$. We write $\Gamma\Leftrightarrow\Delta$ if both
$\Gamma\Rightarrow\Delta$ and $\Delta\Rightarrow\Gamma$ are provable
in $\mathfrak{Bt}$.

A sequent rule is said to be derivable in $\mathfrak{Bt}$ if, whenever
all its premises are provable in $\mathfrak{Bt}$, so is its conclusion.
A double line indicates that the corresponding rules are derivable in
both directions.

\begin{prop}\label{P66}
The following rules are derivable in $\mathfrak{Bt}$.
\end{prop}

\begin{multicols}{2}
\begin{itemize}

  \item[] $(CUT)\,\,\,\,
  \displaystyle
  \frac{\Gamma\Rightarrow\Delta,\alpha
  \qquad
  \alpha,\Sigma\Rightarrow\Pi}
  {\Gamma,\Sigma\Rightarrow\Delta,\Pi}$

  \item[] $(m\G)\,\,\,\,
  \displaystyle
  \frac{\Gamma\Rightarrow\alpha}
  {\G\Gamma\Rightarrow\G\alpha}$

  \item[] $(m\H)\,\,\,\,
  \displaystyle
  \frac{\Gamma\Rightarrow\alpha}
  {\H\Gamma\Rightarrow\H\alpha}$

  \item[] $(m\F)\,\,\,\,
  \displaystyle
  \frac{\alpha\Rightarrow\Delta}
  {\F\alpha\Rightarrow\F\Delta}$

  \item[] $(m\P)\,\,\,\,
  \displaystyle
  \frac{\alpha\Rightarrow\Delta}
  {\P\alpha\Rightarrow\P\Delta}$

  \item[] $(\G)\,\,\,\,\,\,\,\,\,\,\,\,
  \displaystyle
  \frac{\Rightarrow\alpha}
  {\Rightarrow\G\alpha}$

  \item[] $(\H)\,\,\,\,\,\,\,\,\,\,\,\,
  \displaystyle
  \frac{\Rightarrow\alpha}
  {\Rightarrow\H\alpha}$

  \item[] $(Ad_{_{\P\G}})\,\,\,
  \begin{array}{c}
  \P\alpha\Rightarrow\beta\\
  \hline
  \hline
  \alpha\Rightarrow\G\beta
  \end{array}$

  \item[] $(Ad_{_{\F\H}})\,\,\,
  \begin{array}{c}
  \F\alpha\Rightarrow\beta\\
  \hline
  \hline
  \alpha\Rightarrow\H\beta
  \end{array}$

\end{itemize}
\end{multicols}

\begin{proof}
The rule $(CUT)$ follows from $(cut)$ and the weakening rules. 
The rules $(m\G)$ and $(m\H)$ follow from $(\G^{\ast})$ and
$(\H^{\ast})$, respectively, by taking $\Delta=\emptyset$. Similarly,
$(m\F)$ and $(m\P)$ follow from $({\ast}\F)$ and $({\ast}\P)$,
respectively, by taking $\Gamma=\emptyset$. The rules $(\G)$ and $(\H)$
are obtained from $(m\G)$ and $(m\H)$, respectively, by taking
$\Gamma=\emptyset$.

It remains to consider the adjunction rules. We prove
$(Ad_{_{\P\G}})$; the proof of $(Ad_{_{\F\H}})$ is analogous.
First, we show that
$\displaystyle
\frac{\P\alpha\Rightarrow\beta}
{\alpha\Rightarrow\G\beta}$
is derivable.

\begin{prooftree}
\AxiomC{$\alpha\Rightarrow\alpha$}
\LeftLabel{\small{$(\G\P)$}}
\UnaryInfC{$\alpha\Rightarrow\G\P\alpha$}
\AxiomC{$\P\alpha\Rightarrow\beta$}
\LeftLabel{\small{$(m\G)$}}
\UnaryInfC{$\G\P\alpha\Rightarrow\G\beta$}
\LeftLabel{\small{$(cut)$}}
\BinaryInfC{$\alpha\Rightarrow\G\beta$}
\end{prooftree}

Conversely, we show that
$\displaystyle
\frac{\alpha\Rightarrow\G\beta}
{\P\alpha\Rightarrow\beta}$
is derivable.

\begin{prooftree}
\AxiomC{$\alpha\Rightarrow\G\beta$}
\LeftLabel{\small{$(m\P)$}}
\UnaryInfC{$\P\alpha\Rightarrow\P\G\beta$}
\AxiomC{$\beta\Rightarrow\beta$}
\LeftLabel{\small{$(\P\G)$}}
\UnaryInfC{$\P\G\beta\Rightarrow\beta$}
\LeftLabel{\small{$(cut)$}}
\BinaryInfC{$\P\alpha\Rightarrow\beta$}
\end{prooftree}
\end{proof}

\begin{prop}\label{P67}
The following sequents are provable in $\mathfrak{Bt}$.
\begin{enumerate}
  \item[\textbf{(1)}]
  $\F\neg\alpha\Leftrightarrow\neg\G\alpha$ and
  $\P\neg\alpha\Leftrightarrow\neg\H\alpha$;

  \item[\textbf{(2)}]
  $\G(\alpha\wedge\beta)\Leftrightarrow
  \G\alpha\wedge\G\beta$ and
  $\H(\alpha\wedge\beta)\Leftrightarrow
  \H\alpha\wedge\H\beta$;

  \item[\textbf{(3)}]
  $\alpha\Rightarrow\G\P\alpha$ and
  $\alpha\Rightarrow\H\F\alpha$;

  \item[\textbf{(4)}] $\F\alpha\Leftrightarrow\neg\G\neg\alpha$ and $\P\alpha\Leftrightarrow\neg\H\neg\alpha$;

  \item[\textbf{(5)}] $\alpha\Rightarrow\G\neg\H\neg\alpha$ and
  $\alpha\Rightarrow\H\neg\G\neg\alpha$;

  \item[\textbf{(6)}]$\G(\alpha\vee\beta)\Rightarrow\F\alpha\vee \G\beta$ and
  $\H(\alpha\vee\beta)\Rightarrow\P\alpha\vee \H\beta$;

  \item[\textbf{(7)}]
  $\G(\neg\alpha\vee\beta)\Rightarrow\neg\G\alpha\vee \G\beta$ and
  $\H(\neg\alpha\vee\beta)\Rightarrow\neg\H\alpha\vee \H\beta$.
\end{enumerate}
\end{prop}

\begin{proof}

\noindent\textbf{(1)} We prove $\F\neg\alpha\Leftrightarrow\neg\G\alpha.$

\begin{center}
\begin{minipage}{0.43\textwidth}
\begin{prooftree}
\AxiomC{$\alpha\Rightarrow\alpha$}
\LeftLabel{\small{$(\neg\Rightarrow)$}}
\UnaryInfC{$\alpha,\neg\alpha\Rightarrow$}
\LeftLabel{\small{$({\ast}\F)$}}
\UnaryInfC{$\G\alpha,\F\neg\alpha\Rightarrow$}
\LeftLabel{\small{$(\Rightarrow\neg)$}}
\UnaryInfC{$\F\neg\alpha\Rightarrow\neg\G\alpha$}
\end{prooftree}
\end{minipage}
\begin{minipage}{0.43\textwidth}
\begin{prooftree}
\AxiomC{$\alpha\Rightarrow\alpha$}
\LeftLabel{\small{$(\Rightarrow\neg)$}}
\UnaryInfC{$\Rightarrow\neg\alpha,\alpha$}
\LeftLabel{\small{$(\G^{\ast})$}}
\UnaryInfC{$\Rightarrow\F\neg\alpha,\G\alpha$}
\LeftLabel{\small{$(\neg\Rightarrow)$}}
\UnaryInfC{$\neg\G\alpha\Rightarrow\F\neg\alpha$}
\end{prooftree}
\end{minipage}
\end{center}

The proof of $\P\neg\alpha\Leftrightarrow\neg\H\alpha$ is analogous. 

\medskip

\noindent\textbf{(2)} We prove $\G(\alpha\wedge\beta)\Leftrightarrow \G\alpha\wedge\G\beta.$

\begin{prooftree}
\AxiomC{$\alpha\Rightarrow\alpha$}
\LeftLabel{\small{$(we)_l$}}
\UnaryInfC{$\alpha,\beta\Rightarrow\alpha$}
\LeftLabel{\small{$(\wedge\Rightarrow)$}}
\UnaryInfC{$\alpha\wedge\beta\Rightarrow\alpha$}
\LeftLabel{\small{$(m\G)$}}
\UnaryInfC{$\G(\alpha\wedge\beta)\Rightarrow \G\alpha$}

\AxiomC{$\beta\Rightarrow\beta$}
\LeftLabel{\small{$(we)_l$}}
\UnaryInfC{$\alpha,\beta\Rightarrow\beta$}
\LeftLabel{\small{$(\wedge\Rightarrow)$}}
\UnaryInfC{$\alpha\wedge\beta\Rightarrow\beta$}
\LeftLabel{\small{$(m\G)$}}
\UnaryInfC{$\G(\alpha\wedge\beta)\Rightarrow \G\beta$}

\LeftLabel{\small{$(\Rightarrow\wedge)$}}
\BinaryInfC{$\G(\alpha\wedge\beta)\Rightarrow
\G\alpha\wedge\G\beta$}
\end{prooftree}

Conversely,

\begin{prooftree}
\AxiomC{$\alpha\Rightarrow\alpha$}
\LeftLabel{\small{$(we)_l$}}
\UnaryInfC{$\alpha,\beta\Rightarrow\alpha$}

\AxiomC{$\beta\Rightarrow\beta$}
\LeftLabel{\small{$(we)_l$}}
\UnaryInfC{$\alpha,\beta\Rightarrow\beta$}

\LeftLabel{\small{$(\Rightarrow\wedge)$}}
\BinaryInfC{$\alpha,\beta\Rightarrow\alpha\wedge\beta$}
\LeftLabel{\small{$(m\G)$}}
\UnaryInfC{$\G\alpha,\G\beta\Rightarrow
\G(\alpha\wedge\beta)$}
\LeftLabel{\small{$(\wedge\Rightarrow)$}}
\UnaryInfC{$\G\alpha\wedge\G\beta\Rightarrow
\G(\alpha\wedge\beta)$}
\end{prooftree}

The proof for $\H$ is analogous.

\medskip

\noindent\textbf{(3)} The sequents in (3) follow immediately from the
identity axiom by applying $(\G\P)$ and $(\H\F)$, respectively.

\medskip

\medskip

\noindent\textbf{(4)} The proof is similar to that of \textbf{(1)}.

\medskip

\noindent\textbf{(5)} The sequents in \textbf{(5)} follow from
\textbf{(3)} and \textbf{(4)} by applying $(m\G)$, $(m\H)$, and
$(CUT)$, respectively.

\medskip

\noindent\textbf{(6)} The sequents in \textbf{(6)} follow directly
from $(\G^{\ast})$ and $(\H^{\ast})$, respectively, together with the
classical rules.

\medskip

\noindent\textbf{(7)} The sequents in \textbf{(7)} follow from
\textbf{(6)} and \textbf{(1)}, together with the classical rules and
$(CUT)$.

\end{proof}

\subsection{Soundness and completeness}\label{S42}

Let $\Gamma,\Delta\subseteq Fm$ be finite sets. A sequent
$\Gamma\Rightarrow\Delta$ of $\mathfrak{Bt}$ is said to be \emph{valid}
if, for every $\A\in\Bt$ and every
$h\in Hom(\mathfrak{Fm},\A_{\La})$,
\[
\bigwedge\{h(\gamma):\gamma\in\Gamma\}
\leq
\bigvee\{h(\delta):\delta\in\Delta\}.
\]
A sequent rule
\[
\frac{\Gamma_1\Rightarrow\Delta_1\qquad\cdots\qquad
\Gamma_k\Rightarrow\Delta_k}
{\Gamma\Rightarrow\Delta}
\]
is said to \emph{preserve validity} if the validity of all its premises
implies the validity of its conclusion.\\
The following proposition is immediate.

\begin{prop}\label{P68}
Let $\Gamma\Rightarrow\Delta$ be a sequent of $\mathfrak{Bt}$. The
following conditions are equivalent:
\begin{itemize}
  \item[\rm (i)] $\Gamma\Rightarrow\Delta$ is valid;
  \item[\rm (ii)] $\Gamma\models_{\Bt}^{\leqslant}\Delta$.
\end{itemize}
\end{prop}

\begin{lem}\label{L69}
All the structural and logical rules of $\mathfrak{Bt}$ preserve validity.
\end{lem}

\begin{proof}
That the weakening rules and the cut rule preserve validity is an immediate consequence of the fact that every $\A\in\Bt$ has a distributive lattice
reduct. Similarly, it is straightforward to check that the rules
$(\wedge\Rightarrow)$, $(\Rightarrow\wedge)$, $(\vee\Rightarrow)$ and
$(\Rightarrow\vee)$ preserve validity. The rules $(\neg\Rightarrow)$ and
$(\Rightarrow\neg)$ preserve validity by the Boolean algebra structure
of every $\A\in\Bt$.

By (t3) and (t7), the rules $(\G\P)$, $(\H\F)$, $(\P\G)$ and $(\F\H)$
preserve validity.

We now verify that $(\G^{\ast})$ preserves validity. Suppose that
$\Gamma\Rightarrow\Delta,\alpha$ is valid, and let $\A\in\Bt$ and
$h\in \operatorname{Hom}(\mathfrak{Fm},\A_{\La})$. Then
\[
\bigwedge\{h(\gamma):\gamma\in\Gamma\}
\leq
\bigvee\{h(\delta):\delta\in\Delta\}\vee h(\alpha).
\]
By the monotonicity of $\G$ and (t8),
\[
\G\left(\bigwedge\{h(\gamma):\gamma\in\Gamma\}\right)
\leq
\G h(\alpha)\vee
\F\left(\bigvee\{h(\delta):\delta\in\Delta\}\right).
\]
Using (t1), (t2), (t5) and (t6), we obtain
\[
\bigwedge\{h(\G\gamma):\gamma\in\Gamma\}
\leq
\bigvee\{h(\F\delta):\delta\in\Delta\}\vee h(\G\alpha).
\]
Hence $\G\Gamma\Rightarrow\F\Delta,\G\alpha$ is valid. The proof for
$(\H^{\ast})$ is analogous.

Finally, suppose that $\Gamma,\alpha\Rightarrow\Delta$ is valid. Then,
for every $\A\in\Bt$ and every
$h\in \operatorname{Hom}(\mathfrak{Fm},\A_{\La})$,
\[
\bigwedge\{h(\gamma):\gamma\in\Gamma\}\wedge h(\alpha)
\leq
\bigvee\{h(\delta):\delta\in\Delta\}.
\]
By (t9) and the monotonicity of $\F$,
\[
\G\left(\bigwedge\{h(\gamma):\gamma\in\Gamma\}\right)
\wedge\F h(\alpha)
\leq
\F\left(\bigvee\{h(\delta):\delta\in\Delta\}\right).
\]
Using again (t1), (t2), (t5) and (t6), it follows that
\[
\bigwedge\{h(\G\gamma):\gamma\in\Gamma\}\wedge h(\F\alpha)
\leq
\bigvee\{h(\F\delta):\delta\in\Delta\}.
\]
Therefore $\G\Gamma,\F\alpha\Rightarrow\F\Delta$ is valid, and thus
$({\ast}\F)$ preserves validity. The proof for $({\ast}\P)$ is
analogous.
\end{proof}

\begin{theo}[Soundness]\label{T610}
Every sequent provable in $\mathfrak{Bt}$ is valid.
\end{theo}
\begin{proof}
It follows from the validity of the axiom of $\mathfrak{Bt}$ and
Lemma~\ref{L69}.
\end{proof}

As an immediate consequence of the Soundness Theorem,
the empty sequent is not provable in $\mathfrak{Bt}$, since it is not
valid in the two-element tense Boolean algebra.

Next, we prove the Completeness Theorem by means of the
Lindenbaum--Tarski algebra. Consider the binary relations $\prec$ and
$\equiv$ on $Fm$ defined by
\[
\alpha\prec\beta
\quad\text{if and only if}\quad
\mathfrak{Bt}\vdash\alpha\Rightarrow\beta,
\]
and
\[
\alpha\equiv\beta
\quad\text{if and only if}\quad
\alpha\prec\beta \text{ and } \beta\prec\alpha.
\]

By the identity axiom and the structural rules, $\prec$ is a preorder,
and hence $\equiv$ is an equivalence relation. Using the rules
$(\wedge\Rightarrow)$, $(\Rightarrow\wedge)$,
$(\vee\Rightarrow)$, $(\Rightarrow\vee)$,
$(\neg\Rightarrow)$ and $(\Rightarrow\neg)$, it is straightforward to
verify that $\equiv$ is compatible with $\wedge$, $\vee$ and $\neg$.
The compatibility with $\G$, $\H$, $\F$ and $\P$ follows from the
derivable rules $(m\G)$, $(m\H)$, $(m\F)$ and $(m\P)$, respectively.
Hence, $\equiv$ is a congruence on $\mathfrak{Fm}$.

For each $\alpha\in Fm$, we denote by $[\alpha]_{\equiv}$ the
equivalence class of $\alpha$ modulo $\equiv$. The order on
$Fm/{\equiv}$ induced by $\prec$ is given by
\[
[\alpha]_{\equiv}\leq[\beta]_{\equiv}
\quad\text{if and only if}\quad
\mathfrak{Bt}\vdash\alpha\Rightarrow\beta.
\]
We define
\[
\mathbf{1}
=
\{\alpha\in Fm:\mathfrak{Bt}\vdash\,\Rightarrow\alpha\}.
\]

\begin{lem}\label{L610}
The set $\mathbf{1}$ is an equivalence class of $\equiv$. Moreover, if
$\mathbf{0}=\neg\mathbf{1}$, then, for every $\alpha\in Fm$,
\[
\mathbf{0}\leq[\alpha]_{\equiv}\leq\mathbf{1}.
\]
\end{lem}
\begin{proof}
The set $\mathbf{1}$ is nonempty, since, for any propositional variable
$v$, the formula $v\vee\neg v$ is provable in $\mathfrak{Bt}$.

If $\alpha,\beta\in\mathbf{1}$, then
$\Rightarrow\alpha$ and $\Rightarrow\beta$ are provable. By $(we)_l$,
$\alpha\Rightarrow\beta$ and $\beta\Rightarrow\alpha$
are provable, and hence $\alpha\equiv\beta$. Conversely, let
$\alpha\in\mathbf{1}$ and suppose that $\alpha\equiv\beta$. Since
$\alpha\Rightarrow\beta$ and $\Rightarrow\alpha$ are provable, by
$(we)_r$ followed by $(cut)$ we obtain $\Rightarrow\beta$. Thus
$\beta\in\mathbf{1}$. Therefore, $\mathbf{1}$ is an equivalence class
of $\equiv$.

Let $\mathbf{0}=\neg\mathbf{1}$. Since $\equiv$ is compatible with
$\neg$, $\mathbf{0}$ is also an equivalence class. Moreover, for any
propositional variable $v$,
$\mathbf{1}=[v\vee\neg v]_{\equiv}$ and
$\mathbf{0}=[v\wedge\neg v]_{\equiv}$.

For every $\alpha\in Fm$, the sequents
$v\wedge\neg v\Rightarrow\alpha$ and
$\alpha\Rightarrow v\vee\neg v$ are provable in $\mathfrak{Bt}$.
Hence,
$\mathbf{0}\leq[\alpha]_{\equiv}\leq\mathbf{1}$.
\end{proof}
The operations on $Fm/{\equiv}$ are defined in the usual way by
\[
[\alpha]_{\equiv}\wedge[\beta]_{\equiv}
=[\alpha\wedge\beta]_{\equiv},
\qquad
[\alpha]_{\equiv}\vee[\beta]_{\equiv}
=[\alpha\vee\beta]_{\equiv},
\qquad
\neg[\alpha]_{\equiv}=[\neg\alpha]_{\equiv},
\]

and
\[
\#[\alpha]_{\equiv}=[\#\alpha]_{\equiv},
\qquad\text{for }\#\in\{\G,\H,\F,\P\}.
\]

\begin{lem}\label{L611}
The structure $\mathfrak{Fm}_{\equiv}
=
\langle Fm/{\equiv},
\wedge,\vee,\neg,\G,\H,\F,\P,\mathbf{0},\mathbf{1}\rangle$ is a tense Boolean algebra.
\end{lem}

\begin{proof}
It is straightforward to verify, using the rules for $\wedge$, $\vee$
and $\neg$, together with Lemma~\ref{L610}, that the reduct $\langle Fm/{\equiv},
\wedge,\vee,\neg,\mathbf{0},\mathbf{1}\rangle$ is a Boolean algebra. It remains to verify the conditions for the tense operators. Conditions
$(t0)$, $(t2)$ and $(t3)$ follow from \textbf{(1)}, \textbf{(2)} and
\textbf{(3)} of Proposition~\ref{P67}, respectively.

To verify $(t1)$, let $[\alpha]_{\equiv}=\mathbf{1}$. Then $\alpha$ is
provable in $\mathfrak{Bt}$. By the derivable rules $(\G)$ and $(\H)$,
the formulas $\G\alpha$ and $\H\alpha$ are also provable in
$\mathfrak{Bt}$. Hence $[\G\alpha]_{\equiv}=\mathbf{1}$ and $[\H\alpha]_{\equiv}=\mathbf{1}$

Therefore, $\G\mathbf{1}
=\G[\alpha]_{\equiv}
=[\G\alpha]_{\equiv}
=\mathbf{1},$ and similarly, $\H\mathbf{1}
=\H[\alpha]_{\equiv}
=[\H\alpha]_{\equiv}
=\mathbf{1}.$
Thus $(t1)$ holds, and consequently
$\mathfrak{Fm}_{\equiv}$ is a tense Boolean algebra.
\end{proof}

\begin{lem}\label{L612}
Let $\Gamma,\Delta\subseteq Fm$ be finite non-empty sets. Then the
following conditions are equivalent:
\begin{itemize}
    \item[\rm{(i)}] $\mathfrak{Bt}\vdash\Gamma\Rightarrow\Delta$;
    \item[\rm{(ii)}] $\mathfrak{Bt}\vdash
    \bigwedge\Gamma\Rightarrow\bigvee\Delta$.
\end{itemize}
\end{lem}
\begin{proof}
The implication (i)$\Rightarrow$(ii) follows by successive applications
of $(\wedge\Rightarrow)$ and $(\Rightarrow\vee)$. Conversely, suppose that
$\mathfrak{Bt}\vdash\bigwedge\Gamma\Rightarrow\bigvee\Delta.$
Using the identity axiom, $(we)_l$ and
$(\Rightarrow\wedge)$, we obtain
$\mathfrak{Bt}\vdash\Gamma\Rightarrow\bigwedge\Gamma.$
Similarly, using the identity axiom, $(we)_r$ and
$(\vee\Rightarrow)$, we obtain
$\mathfrak{Bt}\vdash\bigvee\Delta\Rightarrow\Delta.$ Then,
\begin{prooftree}
\AxiomC{$\Gamma\Rightarrow\bigwedge\Gamma$}
\AxiomC{$\bigwedge\Gamma\Rightarrow\bigvee\Delta$}
\LeftLabel{\small{$(CUT)$}}
\BinaryInfC{$\Gamma\Rightarrow\bigvee\Delta$}
\AxiomC{$\bigvee\Delta\Rightarrow\Delta$}
\LeftLabel{\small{$(CUT)$}}
\BinaryInfC{$\Gamma\Rightarrow\Delta$}
\end{prooftree}
\end{proof}
We now consider the particular cases in which exactly one of $\Gamma$ and $\Delta$ is empty.

\begin{cor}\label{C613}
Let $\Gamma,\Delta\subseteq Fm$ be finite non-empty sets. Then:
\begin{enumerate}
    \item[\rm (i)]
    $\mathfrak{Bt}\vdash\Gamma\Rightarrow$\,\,
    if and only if\,\,
    $\mathfrak{Bt}\vdash\bigwedge\Gamma\Rightarrow$;

    \item[\rm (ii)]
    $\mathfrak{Bt}\vdash\,\Rightarrow\Delta$\,\,
    if and only if\,\,
    $\mathfrak{Bt}\vdash\,\Rightarrow\bigvee\Delta$.
\end{enumerate}
\end{cor}

\begin{proof}
The left-to-right implications follow by successive applications of
$(\wedge\Rightarrow)$ and $(\Rightarrow\vee)$, respectively. Conversely, using the identity axiom, $(we)_l$ and
$(\Rightarrow\wedge)$, we obtain $\mathfrak{Bt}\vdash\Gamma\Rightarrow\bigwedge\Gamma,$
while, using the identity axiom, $(we)_r$ and
$(\vee\Rightarrow)$, we obtain $\mathfrak{Bt}\vdash\bigvee\Delta\Rightarrow\Delta.$
In each case, the result follows by an application of $(CUT)$.
\end{proof}

\begin{theo}[Completeness]\label{T613}
Every valid sequent of $\mathfrak{Bt}$ is $\mathfrak{Bt}$-provable.
\end{theo}

\begin{proof}
We prove the contrapositive. Let $\Gamma\Rightarrow\Delta$ be a sequent
of $\mathfrak{Bt}$ such that
$\mathfrak{Bt}\not\vdash\Gamma\Rightarrow\Delta$.
Consider the map
$v:Var\longrightarrow Fm/{\equiv}$ defined by
$v(p)=[p]_{\equiv}$. It extends uniquely to a homomorphism
$\overline{v}:\mathfrak{Fm}\longrightarrow
(\mathfrak{Fm}_{\equiv})_{\La}$.
It is straightforward to prove, by induction on the complexity of
$\alpha$, that $\overline{v}(\alpha)=[\alpha]_{\equiv}$ for every $\alpha\in Fm$.

Suppose first that $\Gamma$ and $\Delta$ are non-empty. Let $\gamma=\bigwedge\Gamma$ and $\delta=\bigvee\Delta.$ By Lemma~\ref{L612}, $\mathfrak{Bt}\not\vdash\gamma\Rightarrow\delta,$ and therefore $[\gamma]_{\equiv}\not\leq[\delta]_{\equiv}.$ Hence
\[
\bigwedge\{\overline{v}(\alpha):\alpha\in\Gamma\}
\not\leq
\bigvee\{\overline{v}(\beta):\beta\in\Delta\},
\]
and thus $\Gamma\Rightarrow\Delta$ is not valid.

Suppose now that $\Gamma=\emptyset$ and $\Delta\neq\emptyset$, and let
$\delta=\bigvee\Delta$. By Corollary~\ref{C613}, $\mathfrak{Bt}\not\vdash\,\Rightarrow\delta.$ Hence $[\delta]_{\equiv}\neq\mathbf{1}$. Since $\mathbf{1}$ is the
greatest element of $\mathfrak{Fm}_{\equiv}$,
\[
\mathbf{1}\not\leq[\delta]_{\equiv}
=
\bigvee\{\overline{v}(\beta):\beta\in\Delta\}.
\]
Therefore, $\Rightarrow\Delta$ is not valid.

Suppose that $\Gamma\neq\emptyset$ and $\Delta=\emptyset$, and let
$\gamma=\bigwedge\Gamma$. By Corollary~\ref{C613}, $\mathfrak{Bt}\not\vdash\gamma\Rightarrow.$
Since
\[
\mathfrak{Bt}\vdash\gamma\Rightarrow
\quad\text{if and only if}\quad
[\gamma]_{\equiv}\leq\mathbf{0},
\]

we have
\[
\bigwedge\{\overline{v}(\alpha):\alpha\in\Gamma\}
=
[\gamma]_{\equiv}\not\leq\mathbf{0}.
\]
Therefore, $\Gamma\Rightarrow$ is not valid.
Finally, if $\Gamma=\Delta=\emptyset$, the result follows immediately
from $\mathbf{1}\not\leq\mathbf{0}$ in $\mathfrak{Fm}_{\equiv}$.
Hence, $\Gamma\Rightarrow\Delta$ is not valid.
\end{proof}

From Proposition~\ref{P68} and the Soundness and Completeness Theorems,
we immediately obtain the following result.

\begin{cor}
The following conditions are equivalent:
\begin{itemize}
  \item[\rm{(i)}] $\mathfrak{Bt}\vdash\Gamma\Rightarrow\Delta$;
  \item[\rm{(ii)}] $\Gamma\Rightarrow\Delta$ is valid;
  \item[\rm{(iii)}] $\Gamma\models_{\Bt}^{\leqslant}\Delta$.
\end{itemize}
\end{cor}

Since $K_t$ is sound and complete with respect to tense Boolean algebras,
the Soundness and Completeness Theorems established above yield the
following result.

\begin{theo}
$\mathfrak{Bt}$ is an alternative syntactic presentation of $K_t$.
\end{theo}

Although the preceding result follows immediately from the algebraic
semantics, it can also be established syntactically. We give a direct
syntactic proof below.

Let $Fm_{\mathfrak{Bt}}$ and $Fm_{K_t}$ denote the sets of formulas of
$\mathfrak{Bt}$ and $K_t$, respectively. We define a translation
$t:Fm_{\mathfrak{Bt}}\to Fm_{K_t}$ recursively as follows:
\[
\begin{aligned}
t(p)&=p,\\
t(\neg\alpha)&=\neg t(\alpha),\\
t(\alpha\vee\beta)&=\neg t(\alpha)\to t(\beta),\\
t(\alpha\wedge\beta)&=\neg\bigl(t(\alpha)\to\neg t(\beta)\bigr),\\
t(\G\alpha)&=\G t(\alpha),\\
t(\H\alpha)&=\H t(\alpha),\\
t(\F\alpha)&=\neg\G\neg t(\alpha),\\
t(\P\alpha)&=\neg\H\neg t(\alpha).
\end{aligned}
\]
In what follows, we use the usual abbreviations in $K_t$
\[
\varphi\vee\psi:=\neg\varphi\to\psi,
\qquad
\varphi\wedge\psi:=\neg(\varphi\to\neg\psi),
\]
and
\[
\F\varphi:=\neg\G\neg\varphi,
\qquad
\P\varphi:=\neg\H\neg\varphi.
\]
With these conventions, the definition of $t$ can equivalently be
written as
\[
t(\alpha\circ\beta)=t(\alpha)\circ t(\beta),
\qquad
\circ\in\{\wedge,\vee\},
\]
and
\[
t(\#\alpha)=\# t(\alpha),
\qquad
\#\in\{\G,\H,\F,\P\}.
\]
Conversely, we define a translation
$s:Fm_{K_t}\to Fm_{\mathfrak{Bt}}$ by
\[
\begin{aligned}
s(p)&=p,\\
s(\neg\alpha)&=\neg s(\alpha),\\
s(\alpha\to\beta)&=\neg s(\alpha)\vee s(\beta),\\
s(\G\alpha)&=\G s(\alpha),\\
s(\H\alpha)&=\H s(\alpha).
\end{aligned}
\]
The translations defined above are mutually inverse up to provable
equivalence, as stated in the following lemma.

\begin{lem}\label{Ltrans}
For every $\alpha\in Fm_{\mathfrak{Bt}}$,
\[
\mathfrak{Bt}\vdash\alpha\Rightarrow s(t(\alpha))
\qquad\text{and}\qquad
\mathfrak{Bt}\vdash s(t(\alpha))\Rightarrow\alpha.
\]
Moreover, for every $\varphi\in Fm_{K_t}$,
\[
K_t\vdash\varphi\to t(s(\varphi))
\qquad\text{and}\qquad
K_t\vdash t(s(\varphi))\to\varphi.
\]
\end{lem}

\begin{proof}
The proof is by induction on the complexity of the formulas. The
Boolean cases follow from classical logic. For $\mathfrak{Bt}$, the
cases involving $\G$ and $\H$ follow from $(m\G)$ and $(m\H)$,
respectively, while the cases involving $\F$ and $\P$ follow from the
induction hypothesis, Proposition~\ref{P67}\textbf{(1)}, and the rules
for negation. For $K_t$, the cases involving $\G$ and $\H$ follow from
the induction hypothesis, the necessitation rules, and the corresponding
$K$-axioms.
\end{proof}

\begin{lem}\label{Ltrans-prov1}
For every $\varphi\in Fm_{K_t}$, if $K_t\vdash\varphi$, then
$\mathfrak{Bt}\vdash\,\Rightarrow s(\varphi).$
\end{lem}
\begin{proof}
The proof is by induction on a proof of $\varphi$ in $K_t$. The translations under $s$ of the classical axioms are provable in the
classical fragment of $\mathfrak{Bt}$. Let us consider the first
$K$-axiom
\[
\G(\alpha\to\beta)\to(\G\alpha\to\G\beta).
\]
Its translation under $s$ is
\[
\neg\G(\neg s(\alpha)\vee s(\beta))
\vee
\bigl(\neg\G s(\alpha)\vee\G s(\beta)\bigr).
\]
By Proposition~\ref{P67}\textbf{(7)},
\[
\mathfrak{Bt}\vdash
\G(\neg s(\alpha)\vee s(\beta))
\Rightarrow
\neg\G s(\alpha)\vee\G s(\beta),
\]
and hence, by the classical rules,
\[
\mathfrak{Bt}\vdash\,
\Rightarrow
\neg\G(\neg s(\alpha)\vee s(\beta))
\vee
\bigl(\neg\G s(\alpha)\vee\G s(\beta)\bigr).
\]
The case of the corresponding axiom for $\H$ is analogous. The
translations of the remaining two tense axioms are provable by
Proposition~\ref{P67}\textbf{(5)}.

Suppose now that $\varphi$ is obtained by modus ponens from $\psi$ and
$\psi\to\varphi$. By the induction hypothesis,
\[
\mathfrak{Bt}\vdash\,\Rightarrow s(\psi)
\qquad\text{and}\qquad
\mathfrak{Bt}\vdash\,\Rightarrow
\neg s(\psi)\vee s(\varphi).
\]
Then $\mathfrak{Bt}\vdash\,\Rightarrow s(\varphi)$ follows from the
classical rules and $(CUT)$.
Finally, if $\varphi$ is obtained by one of the necessitation rules,
the result follows immediately from the induction hypothesis by
applying the corresponding derivable rule $(\G)$ or $(\H)$.
\end{proof}

\begin{lem}\label{Ltrans-prov2}
For every $\alpha\in Fm_{\mathfrak{Bt}}$, if
$\mathfrak{Bt}\vdash\,\Rightarrow\alpha$, then
$K_t\vdash t(\alpha)$.
\end{lem}

In order to prove Lemma~\ref{Ltrans-prov2}, we extend the translation
$t$ from formulas to sequents. For every nonempty sequent
$\Gamma\Rightarrow\Delta$, define
\[
\tau(\Gamma\Rightarrow\Delta)=
\begin{cases}
t\left(\bigwedge\Gamma\right)\to
t\left(\bigvee\Delta\right),
& \text{if } \Gamma\neq\emptyset \text{ and } \Delta\neq\emptyset,\\[2mm]
t\left(\bigvee\Delta\right),
& \text{if } \Gamma=\emptyset \text{ and } \Delta\neq\emptyset,\\[2mm]
\neg t\left(\bigwedge\Gamma\right),
& \text{if } \Gamma\neq\emptyset \text{ and } \Delta=\emptyset.
\end{cases}
\]
Notice that $\tau$ need not be defined on the empty sequent
$\Rightarrow$, since it is not provable in $\mathfrak{Bt}$.
The following result provides the key step in the proof of
Lemma~\ref{Ltrans-prov2}.

\begin{prop}\label{Ptrans-seq}
For every nonempty sequent $\Gamma\Rightarrow\Delta$, if $\mathfrak{Bt}\vdash\Gamma\Rightarrow\Delta,$
then $K_t\vdash\tau(\Gamma\Rightarrow\Delta).$
\end{prop}

Notice that Lemma~\ref{Ltrans-prov2} follows immediately from
Proposition~\ref{Ptrans-seq} by taking $\Gamma=\emptyset$ and
$\Delta=\{\alpha\}$, since in this case
$\tau(\Rightarrow\alpha)=t(\alpha)$. We now prove
Proposition~\ref{Ptrans-seq}.

\begin{proof}
The proof is by induction on a derivation of
$\Gamma\Rightarrow\Delta$ in $\mathfrak{Bt}$. The identity axiom is
immediate, since $\tau(\alpha\Rightarrow\alpha)
=
t(\alpha)\to t(\alpha)$
and $K_t\vdash t(\alpha)\to t(\alpha)$.
We consider the rule
\[
(cut)\qquad
\frac{\Gamma\Rightarrow\Delta,\alpha
      \qquad
      \alpha,\Gamma\Rightarrow\Delta}
     {\Gamma\Rightarrow\Delta}.
\]
Assume that $\Gamma\neq\emptyset$ and $\Delta\neq\emptyset$. By the
induction hypothesis applied to the premises, we have
\[
K_t\vdash
t\left(\bigwedge\Gamma\right)\to
\left(\neg t\left(\bigvee\Delta\right)\to t(\alpha)\right),
\qquad
K_t\vdash
\neg\left(t(\alpha)\to\neg t\left(\bigwedge\Gamma\right)\right)
\to t\left(\bigvee\Delta\right).
\]
Using the abbreviations for $\vee$ and $\wedge$ in $K_t$, these
formulas can be written as
\[
K_t\vdash
t\left(\bigwedge\Gamma\right)\to
\left(t\left(\bigvee\Delta\right)\vee t(\alpha)\right),
\qquad
K_t\vdash
\left(t(\alpha)\wedge t\left(\bigwedge\Gamma\right)\right)
\to t\left(\bigvee\Delta\right).
\]
Put
\[
A=t\left(\bigwedge\Gamma\right),\qquad
B=t\left(\bigvee\Delta\right),\qquad
C=t(\alpha).
\]
Then
\[
K_t\vdash A\to(B\vee C)
\qquad\text{and}\qquad
K_t\vdash(C\wedge A)\to B.
\]
On the other hand, it is easy to check that the following is a
theorem of $K_t$:
\[
K_t\vdash
(A\to(B\vee C))
\to
\left(
((C\wedge A)\to B)
\to(A\to B)
\right).
\]
Therefore,
\[
K_t\vdash A\to B,
\]
that is,
\[
K_t\vdash\tau(\Gamma\Rightarrow\Delta).
\]
The cases in which one of the contexts is empty are treated
similarly.

The weakening rules and the rules for $\neg$, $\wedge$, and $\vee$
are handled in the same way, using the corresponding theorems of the
classical fragment of $K_t$.

Let us consider now the temporal rule
\[
(\G^{\ast})\qquad
\frac{\Gamma\Rightarrow\Delta,\alpha}
     {\G\Gamma\Rightarrow\F\Delta,\G\alpha}.
\]
Assume that $\Gamma\neq\emptyset$ and $\Delta\neq\emptyset$. By the
induction hypothesis applied to the premise, using the abbreviation
for $\vee$ in $K_t$, we have
\[
K_t\vdash A\to(B\vee C),
\]
where
\[
A=t\left(\bigwedge\Gamma\right),\qquad
B=t\left(\bigvee\Delta\right),\qquad
C=t(\alpha).
\]
By $\G$-necessitation and the $\G$-axiom,
\[
K_t\vdash\G A\to\G(B\vee C).
\]
It is easy to check that the following is a theorem of $K_t$:
\[
K_t\vdash
\G(B\vee C)\to(\F B\vee\G C).
\]
Therefore,
\[
K_t\vdash
\G A\to(\F B\vee\G C).
\]
A straightforward induction on the number of formulas in the
contexts yields
\[
K_t\vdash
\G A\leftrightarrow t\left(\bigwedge\G\Gamma\right)
\qquad\text{and}\qquad
K_t\vdash
\F B\leftrightarrow t\left(\bigvee\F\Delta\right).
\]
Therefore,
\[
K_t\vdash
t\left(\bigwedge\G\Gamma\right)
\to
\left(
t\left(\bigvee\F\Delta\right)\vee t(\G\alpha)
\right),
\]
from which it follows that
\[
K_t\vdash
\tau(\G\Gamma\Rightarrow\F\Delta,\G\alpha).
\]
The cases in which one or both contexts are empty are treated
similarly. The case of $(\H^{\ast})$ is analogous. The remaining temporal rules
$({\ast}\F)$, $({\ast}\P)$, $(\P\G)$, $(\F\H)$, $(\G\P)$, and
$(\H\F)$ are treated similarly, using the corresponding theorems of
$K_t$. This completes the induction.
\end{proof}
To conclude this section, we observe that the decidability of $K_t$,
together with the equivalence established above, immediately yields
the following result.

\begin{theo}
$\mathfrak{Bt}$ is decidable.
\end{theo}

\section{Kashima's system}

Let us consider the language $\La$ extended with the following auxiliary symbols: (, ), $^{\P}\!\lbrace$, $^{\F}\!\lbrace$, $\rbrace$, comma.\\
We denote by $Var$ the set of all propositional variables. The notion of {\em formula} is defined as follows:
\begin{itemize}
  \item[1.] If $v$ is a propositional variable, then $v$ and $\neg v$ are formulas.
  \item[2.] If $\alpha$ and $\beta$ are formulas, then $(\alpha \wedge \beta), (\alpha \vee \beta), \H \alpha, \G \alpha, \P \alpha$ and $\F \alpha$ are formulas. 
\end{itemize}
When there is no ambiguity, parentheses will be omitted.

We denote by $\Fs$ the set of all formulas. Note that negation symbol $\neg$ is allowed only in front of propositional variables. We use letters $u, v, \dots$ for propositional variables, when convenient, we use subscripts; and letters $\alpha, \beta,\dots$ for formulas.\\
Kashima's sequents are defined inductively as follows:
\begin{itemize}
  \item[1.] A formula is a sequent.
  \item[2.] If $\Gamma$ is a sequent, then $^\P\!\lbrace \Gamma \rbrace$ and $^\F\!\lbrace \Gamma \rbrace$ are sequents.
  \item[3.] If $n \geq 0$ and each $\Gamma_i$ ($1 \leq i \leq n$) is a sequent, then the sequence $\Gamma_1, \Gamma_2, \dots , \Gamma_n$ is a sequent.
\end{itemize}
In what follows, we use letters $\Gamma, \Delta, \dots$ for sequents. For a sequent $\Gamma$ the {\em meaning formula} of $\Gamma$, $\Gamma ^*$, is defined inductively as follows:
\begin{itemize}
  \item[1.] if $\Gamma = \alpha$ then $\Gamma ^* = \alpha$;
  \item[2.] $(^\P\!\lbrace \Delta \rbrace)^*=\H \Delta ^*$ and $(^\F\!\lbrace \Delta \rbrace)^*=\G \Delta ^*$;
  \item[3.] if $n > 0$, then $(\Gamma_1, \Gamma_2, \dots , \Gamma_n)^*=\Gamma_1^* \vee \Gamma_2^* \vee \dots \vee \Gamma_n^*$. If $n=0$ then $(\,\, \,)^*=(v_0 \wedge \neg v_0)$.
\end{itemize}

Let $v_0$ be a propositional variable and let $\Gamma$ be a sequent in which there is exactly one occurrence of $v_0$ and $v_0$ is not a subformula of any other formula in $\Gamma$. By  $\Gamma [\Delta]$ we denote the sequent obtained from $\Gamma$ 
by replacing $v_0$ by the sequent $\Delta$. 

Let $SK_t$ be the system defined as follows.

\begin{defi}\rm(Kashima's system for $K_t$ \cite{Ka}) Let $SK_t$ be the sequent system consisting of the following:

\

\noindent{\bf Axiom}
$$v, \neg v$$
\noindent{\bf Rules}

\begin{center}
\AxiomC{$\Gamma , \Delta, \Pi, \Sigma$}
\RightLabel{\rm exchange}
\UnaryInfC{$\Gamma, \Pi , \Delta, \Sigma$}
\DisplayProof \hspace{1cm}
\AxiomC{$\Gamma$}
\RightLabel{\rm weakening}
\UnaryInfC{$\Gamma, \Delta$}
\DisplayProof \hspace{1cm}
\AxiomC{$\Gamma, \alpha, \alpha$}
\RightLabel{\rm contraction}
\UnaryInfC{$\Gamma, \alpha$}
\DisplayProof 
\end{center}

\begin{center}
\AxiomC{$\Gamma, ^\F\!\lbrace \Delta \rbrace$}
\RightLabel{\rm turn}
\UnaryInfC{$^\P\!\lbrace \Gamma \rbrace, \Delta$}
\DisplayProof \hspace{1cm}
\AxiomC{$\Gamma, ^\P\!\lbrace \Delta \rbrace$}
\RightLabel{\rm turn}
\UnaryInfC{$^\F\!\lbrace \Gamma \rbrace, \Delta$}
\DisplayProof \end{center}

\begin{center}
\AxiomC{$\Gamma, \alpha$}
\AxiomC{$\Gamma, \beta$}
\RightLabel{$\wedge$}
\BinaryInfC{$\Gamma, (\alpha \wedge \beta)$}
\DisplayProof \hspace{1cm}
\AxiomC{$\Gamma, \alpha, \beta$}
\RightLabel{$\vee$}
\UnaryInfC{$\Gamma, (\alpha \vee \beta)$}
\DisplayProof \end{center}

\begin{center}
\AxiomC{$\Gamma, ^\P\!\lbrace \alpha \rbrace$}
\RightLabel{$\H$}
\UnaryInfC{$\Gamma, \H \alpha$}
\DisplayProof \hspace{1cm}
\AxiomC{$\Gamma, ^\F\!\lbrace \alpha \rbrace$}
\RightLabel{$\G$}
\UnaryInfC{$\Gamma, \G \alpha$}
\DisplayProof \end{center}

\begin{center}
\AxiomC{$\Gamma, ^\P\!\lbrace \Delta , \alpha \rbrace$}
\RightLabel{$\P$}
\UnaryInfC{$\Gamma, ^\P\!\lbrace \Delta \rbrace, \P \alpha$}
\DisplayProof \hspace{1cm}
\AxiomC{$\Gamma, ^\F\!\lbrace \Delta , \alpha \rbrace$}
\RightLabel{$\F$}
\UnaryInfC{$\Gamma, ^\F\!\lbrace \Delta \rbrace, \F \alpha$}
\DisplayProof \end{center}

\begin{prooftree}
\AxiomC{$\Gamma, \alpha$}
\AxiomC{$\Gamma, \neg \alpha$}
\RightLabel{\rm cut}
\BinaryInfC{$\Gamma$}
\end{prooftree}
where $\neg \alpha$ is defined inductively as follows: $\neg \neg v=v, \neg (\alpha \wedge \beta)=(\neg \alpha \vee \neg \beta), \neg (\alpha \vee \beta)=(\neg \alpha \wedge \neg \beta), \break \neg \H \alpha=\P \neg \alpha, \neg \G \alpha = \F \neg \alpha, \neg \P \alpha= \H \neg \alpha$ and $\neg \F \alpha=\G \neg \alpha$.
\end{defi}

The notion of proof in $SK_t$ is the usual. Besides, if there is a proof in $SK_t$ for the sequent $\Gamma$ we write $SK_t\vdash \Gamma$  and we say that $\Gamma$ is $SK_t$-provable. Also, we say that the formula $\alpha$ is a theorem of $SK_t$ iff the sequent $\alpha$ is $SK_t$-provable. 

\begin{prop}\label{prop4.1} For every formula $\alpha$, we have that $\alpha=\neg \neg \alpha$
\end{prop}
\begin{proof} 
By  induction on the complexity of $\alpha$.
\end{proof} 

\begin{prop}\label{prop4.2} For every formula $\alpha$, it holds that  $SK_t\vdash \alpha,\neg \alpha$.
\end{prop}
\begin{proof} 
We use induction on the complexity of $\alpha$,  $c(\alpha)$.\\
If $c(\alpha)=0$ then $\alpha$ is a propositional variable and the proposition holds trivially.\\
Suppose that $c(\alpha)=n>0$. Let us assume that (I.H.) for every formula $\beta$ such that $c(\beta)<n$ it holds $SK_t\vdash  \beta,\neg \beta$. We analyze different cases.\\
(1) $\alpha=\neg \beta$. We have to see that $SK_t\vdash  \alpha, \neg \alpha$, that is $SK_t\vdash \neg \beta, \neg \neg \beta$. By Proposition \ref{prop4.1} , we have to prove that  $SK_t\vdash \neg \beta, \beta$, but this holds by the (I.H.) and using exchange.\\
(2)  $\alpha=\beta\wedge \gamma$. We have to prove that $SK_t\vdash \beta \wedge \gamma, \neg (\beta \wedge \gamma)$, but by the definition of $\neg$, we have to prove that $SK_T\vdash \beta \wedge \gamma, \neg \beta \vee \neg \gamma$. By the (I.H.) we know that $SK_t\vdash \beta, \neg \beta$ and $SK_t \vdash  \gamma, \neg \gamma$.

\begin{prooftree}
\AxiomC{$\begin{array}{c} \vdots \\ \beta, \neg \beta \end{array}$}
\RightLabel{w}
\UnaryInfC{$\beta, \neg \beta, \neg \gamma$}
\RightLabel{e}
\UnaryInfC{$\neg \beta, \beta,  \neg \gamma$}
\RightLabel{e}
\UnaryInfC{$\neg \beta, \neg \gamma, \beta$}
\AxiomC{$\begin{array}{c} \vdots \\ \gamma, \neg \gamma \end{array}$}
\RightLabel{w}
\UnaryInfC{$\gamma, \neg \gamma, \neg \beta$}
\RightLabel{e}
\UnaryInfC{$\neg \gamma, \gamma,  \neg \beta$}
\RightLabel{e}
\UnaryInfC{$\neg \gamma, \neg \beta, \gamma$}
\RightLabel{e}
\UnaryInfC{$\neg \beta, \neg \gamma, \gamma$}
\RightLabel{$\wedge$}
\BinaryInfC{$\neg \beta, \neg \gamma, \beta \wedge \gamma$}
\RightLabel{e}
\UnaryInfC{$\neg \beta, \beta \wedge \gamma, \neg \gamma$}
\RightLabel{e}
\UnaryInfC{$\beta \wedge \gamma, \neg \beta, \neg \gamma$}
\RightLabel{$\vee$}
\UnaryInfC{$\beta \wedge \gamma, \neg \beta \vee \neg \gamma$}
\end{prooftree}
(3) $\alpha=\beta\vee \gamma$. Analogously to the previous case.\\
(4) $\alpha=\P \beta$. Let us see that $SK_t\vdash \P \beta, \neg \P \beta$, i.e. $SK_t \vdash  \P \beta, \H \neg \beta$. By the (I.H.) we have $SK_t\vdash \beta, \neg \beta$.
\begin{prooftree}
\AxiomC{$\begin{array}{c} \vdots \\ \beta, \neg \beta \end{array}$}
\RightLabel{e}
\UnaryInfC{$\neg \beta, \beta$}
\RightLabel{w}
\UnaryInfC{$\neg \beta, \beta, ^\F\!\{ \}$}
\RightLabel{t}
\UnaryInfC{$^\P\!\{ \neg \beta, \beta \}$}
\RightLabel{$\P$}
\UnaryInfC{$^\P\!\{ \neg \beta \}, \P \beta$}
\RightLabel{e}
\UnaryInfC{$\P \beta, ^\P\!\{ \neg \beta \}$}
\RightLabel{$\H$}
\UnaryInfC{$\P \beta, \H\neg \beta$}
\end{prooftree}
Cases (5) $\alpha=\F \beta$, (6) $\alpha=\H \beta$ and (7) $\alpha=\G \beta$   are proved analogously to the previous case. \end{proof}

\begin{prop}\label{prop4.3} If $SK_t\vdash \alpha, \neg \beta$,  $SK_t \vdash \beta, \neg \alpha$, $SK_t\vdash \beta, \neg \gamma$ and $SK_t \vdash \gamma, \neg \beta$ then $SK_t\vdash \alpha, \neg \gamma$ and $SK_t\vdash \gamma, \neg \alpha$
\end{prop}
\begin{proof} 
From $SK_t \vdash \beta, \neg \gamma$ and $SK_t\vdash \alpha, \neg \beta$ we have
\begin{prooftree}
\AxiomC{$\begin{array}{c} \vdots \\ \beta, \neg \gamma \end{array}$}
\RightLabel{w}
\UnaryInfC{$\beta, \neg \gamma, \alpha$}
\RightLabel{e}
\UnaryInfC{$\neg \gamma, \beta, \alpha$}
\RightLabel{e}
\UnaryInfC{$\neg \gamma, \alpha, \beta$}
\RightLabel{e}
\UnaryInfC{$\alpha, \neg \gamma, \beta$}
\AxiomC{$\begin{array}{c} \vdots \\ \alpha, \neg \beta \end{array}$}
\RightLabel{w}
\UnaryInfC{$\alpha, \neg \beta, \neg \gamma$}
\RightLabel{e}
\UnaryInfC{$\alpha, \neg \gamma, \neg \beta$}
\RightLabel{cut}
\BinaryInfC{$\alpha, \neg \gamma$}
\end{prooftree}
On the other hand, from $SK_t\vdash \beta, \neg \alpha$ and $SK_t\vdash \gamma, \neg \beta$ we have
\begin{prooftree}
\AxiomC{$\begin{array}{c} \vdots \\ \beta, \neg \alpha \end{array}$}
\RightLabel{w}
\UnaryInfC{$\beta, \neg \alpha, \gamma$}
\RightLabel{e}
\UnaryInfC{$\neg \alpha, \beta, \gamma$}
\RightLabel{e}
\UnaryInfC{$\neg \alpha, \gamma, \beta$}
\RightLabel{e}
\UnaryInfC{$\gamma, \neg \alpha, \beta$}
\AxiomC{$\begin{array}{c} \vdots \\ \gamma, \neg \beta \end{array}$}
\RightLabel{w}
\UnaryInfC{$\gamma, \neg \beta, \neg \alpha$}
\RightLabel{e}
\UnaryInfC{$\gamma, \neg \alpha, \neg \beta$}
\RightLabel{cut}
\BinaryInfC{$\gamma, \neg \alpha$}
\end{prooftree}
\end{proof} 

\subsection{Soundness and completeness}
\begin{defi} 
\noindent(1) Let $\alpha \in \Fs$. We say that $\alpha$ is valid in a given tense Boolean algebra $\A$ if for every homomorphism  $h:\mathfrak{F_m}_{SK_t} \to \A_\La$, we have $h(\alpha)=1$ 
. We say that $\alpha$ is valid if it is valid in every tense Boolean algebra. \\
(2) If $\Gamma$ is a sequent, we say that $\Gamma$ is valid if $ \Gamma^*$ is valid.
\end{defi}

\begin{lem} Every axiom of $SK_t$ is valid and every rule in $SK_t$ preserves validity. 
\end{lem}
\begin{proof} Let $h:\Fs \to \A_\La$ be a homomorphism and let $\A \in \mathbb{B}_t$.\\
$v_0,\neg v_0$ is valid. This follows immediately from the fact that $h$ is a homomorphism. \\
Let us show that every rule preserves validity. We will prove only one of the Turn rules (the other one is proved analogously) and the rule $\P$ (the rule $\F$ is proved analogously). The validity of the remaining rules follows directly from the definition of the meaning formula and the fact that $h$ is a homomorphism. \\[2mm]
\textbf{Turn:} Suppose that $\Gamma, ^\P\!\{ \Delta\}$ is valid, then $h((\Gamma, ^\P\!\{ \Delta\})^*)=1$. By Proposition \ref{prop3.1} (t14), we have:
\[{\renewcommand{\arraystretch}{1.4}
\begin{tabular}{rll}
$1$ & $=h((\Gamma, ^\P\!\{ \Delta\})^*)$                        &\\
    & $=h(\Gamma^* \vee \H(\Delta^*))$                        &\\
    & $=h(\Gamma^*) \vee \H(h(\Delta^*))$                     &, \\
    & $=\G(h(\Gamma^*)) \vee h(\Delta^*)$                     &\\
    & $=h(\G(\Gamma^*) \vee \Delta^*)$                        &\\
    & $=h((^\F\!\{\Gamma\}, \Delta)^*)$                         &\\
\end{tabular}}\]
Then, $^F\{\Gamma\}, \Delta$ is valid. \\
[2mm]
\textbf{P:} If $\Gamma, ^\P\!\{ \Delta, \alpha\}$ is valid, then:
$$ {\renewcommand{\arraystretch}{1.3}
\begin{tabular}{rll}
$1$ & $=h((\Gamma, ^\P\!\{ \Delta, \alpha\})^*)$                         & \\
    & $=h(\Gamma^* \vee \H(\Delta^* \vee \alpha) )$                    & \\
    & $=h(\Gamma^*) \vee \H(h(\Delta^*) \vee v(\alpha))$              &\\
\end{tabular}}$$
By Proposition \ref{prop3.1} (t8) we know that  $\H(h(\Delta^*) \vee h(\alpha)) \leq \H(h(\Delta)^*) \vee \P(h(\alpha))$. Then, \break $h(\Gamma^*) \vee \H(h(\Delta^*) \vee h(\alpha)) \leq h(\Gamma^*) \vee \H(h(\Delta^*)) \vee \P(h(\alpha))$ and then $1 \leq h(\Gamma^*) \vee \H(h(\Delta^*)) \vee \P(h(\alpha))$. Therefore
\[ {\renewcommand{\arraystretch}{1.3}
\begin{tabular}{rll}
$1$ & $=h(\Gamma^*) \vee \H(h(\Delta^*)) \vee \P(h(\alpha))$                & \\
    & $=h(\Gamma^* \vee \H(\Delta^*) \vee \P \alpha )$                    & \\
    & $=h((\Gamma, ^\P\!\{\Delta\}, \P \alpha)^*)$                                &\\
\end{tabular}} \]
That is, $\Gamma, ^\P\{\Delta\},\P(\alpha)$ is valid.
\end{proof}

\

\begin{theo} (Correctness) If $SK_t\vdash \Gamma$ then $ \Gamma$ is valid.
\end{theo}
\begin{proof} By induction on the size (number of instances of rule applications) of the proof of $\Gamma$.
\end{proof}

\noindent We define the binary relation $\equiv$ on the set of all formulas as follows. For every $\alpha, \beta\in \Fs$ 

\begin{center} $\alpha \equiv \beta$ \ if and only if \ $SK_t \vdash \alpha, \neg \beta$ and $SK_t \vdash \beta, \neg \alpha$. \end{center}
Then
\begin{prop}\label{prop4.4} $\equiv$ is an equivalence relation. 
\end{prop}
\begin{proof} 
Reflexivity and transitivity follow from propositions  \ref{prop4.2} and \ref{prop4.3}, symmetry is immediate from definition of $\equiv$.
\end{proof} 
\noindent Moreover, 
\begin{prop} $\equiv$ is a congruence relation. That is, if $\alpha\equiv \beta$ and $\gamma\equiv \delta$ then 
\begin{multicols}{2}
\begin{enumerate}[\rm (1)]
\item $\alpha \wedge \gamma\equiv \beta \wedge \delta$, 
\item $\alpha \vee \gamma \equiv \beta \vee \delta$,
\item $\neg \alpha \equiv \neg \beta$,
\item $\G \alpha \equiv \G \beta$,
\item $\H \alpha \equiv \H \beta$,
\item $\F \alpha \equiv \F \beta$,
\item $\P \alpha \equiv \P \beta$.
\end{enumerate}
\end{multicols}
\end{prop}

\begin{proof} We provide three examples; the remaining cases are treated analogously, taking into account the definition of negation for formulas.
\\
[2mm]
Let us see that if $\alpha \equiv \beta$ then $\neg \alpha \equiv \neg \beta$. Notice that to prove $\neg \alpha \equiv \neg \beta$ we need to prove that  $SK_t \vdash  \neg \alpha, \neg \neg \beta$ and $SK_t \vdash  \neg \beta, \neg \neg \alpha$. But this is equivalent to proving $SK_t \vdash \neg \alpha, \beta$ and $SK_t \vdash  \neg \beta, \alpha$ (by Proposition \ref{prop4.1}), and this holds by hypothesis (and using the exchange rule).\\
[2mm]
Suppose that  $\alpha \equiv \beta$ and $\gamma \equiv \delta$. To prove that $\alpha \wedge \gamma \equiv \beta \wedge \delta$ we need to check that  $SK_t \vdash \alpha \wedge \gamma, \neg( \beta  \wedge \delta)$ and $SK_t \vdash \beta \wedge \delta, \neg( \alpha \wedge \gamma)$ which, in turn, is equivalent to $SK_t \vdash \alpha \wedge \gamma, \neg \beta \vee \neg \delta$ and $SK_t \vdash \beta \wedge \delta, \neg \alpha \vee \neg \gamma$. Then, consider the following proofs.
\begin{multicols}{2}
\begin{prooftree}
\AxiomC{$\begin{array}{c} \vdots \\ \alpha, \neg \beta \end{array}$}
\RightLabel{e}
\UnaryInfC{$\neg \beta, \alpha$}
\RightLabel{w}
\UnaryInfC{$\neg \beta, \alpha, \neg \delta$}
\RightLabel{e}
\UnaryInfC{$\neg \beta, \neg \delta, \alpha$}
\AxiomC{$\begin{array}{c} \vdots \\ \gamma, \neg \delta \end{array}$}
\RightLabel{w}
\UnaryInfC{$\gamma, \neg \delta, \neg \beta$}
\RightLabel{e}
\UnaryInfC{$\neg \delta, \gamma, \neg \beta$}
\RightLabel{e}
\UnaryInfC{$\neg \delta, \neg \beta, \gamma$}
\RightLabel{e}
\UnaryInfC{$\neg \beta, \neg \delta, \gamma$}
\RightLabel{$\wedge$}
\BinaryInfC{$\neg \beta, \neg \delta, \alpha \wedge \gamma$}
\RightLabel{e}
\UnaryInfC{$\neg \beta, \alpha \wedge \gamma, \neg \delta$}
\RightLabel{e}
\UnaryInfC{$\alpha \wedge \gamma, \neg \beta, \neg \delta$}
\RightLabel{$\vee$}
\UnaryInfC{$\alpha \wedge \gamma, \neg \beta \vee \neg \delta$}
\end{prooftree}

\begin{prooftree}
\AxiomC{$\begin{array}{c} \vdots \\ \beta, \neg \alpha \end{array}$}
\RightLabel{e}
\UnaryInfC{$\neg \alpha, \beta$}
\RightLabel{w}
\UnaryInfC{$\neg \alpha, \beta, \neg \gamma$}
\RightLabel{e}
\UnaryInfC{$\neg \alpha, \neg \gamma, \beta$}
\AxiomC{$\begin{array}{c} \vdots \\ \delta, \neg \gamma \end{array}$}
\RightLabel{w}
\UnaryInfC{$\delta, \neg \gamma, \neg \alpha$}
\RightLabel{e}
\UnaryInfC{$\neg \gamma, \delta, \neg \alpha$}
\RightLabel{e}
\UnaryInfC{$\neg \gamma, \neg \alpha, \delta$}
\RightLabel{e}
\UnaryInfC{$\neg \alpha, \neg \gamma, \delta$}
\RightLabel{$\wedge$}
\BinaryInfC{$\neg \alpha, \neg \gamma, \beta \wedge \delta$}
\RightLabel{e}
\UnaryInfC{$\neg \alpha, \beta \wedge \delta, \neg \gamma$}
\RightLabel{e}
\UnaryInfC{$\beta \wedge \delta, \neg \alpha, \neg \gamma$}
\RightLabel{$\vee$}
\UnaryInfC{$\beta \wedge \delta, \neg \alpha \vee \neg \gamma$}
\end{prooftree}
\end{multicols}

\noindent Suppose that  $\alpha \equiv \beta$. To prove that $\P \alpha \equiv \P \beta$ is equivalent to proving $SK_t \vdash \P \alpha, \neg \P \beta$ and $SK_t \vdash \P \beta , \neg \P \alpha$ which, in turn, is equivalent to proving $SK_t \vdash  \P \alpha, \H \neg \beta$ and $SK_t \vdash  \P \beta , \H \neg \alpha$. Consider the following proofs.

\begin{multicols}{2}
\begin{prooftree}
\AxiomC{$\begin{array}{c} \vdots \\ \alpha, \neg \beta \end{array}$}
\RightLabel{e}
\UnaryInfC{$\neg \beta, \alpha$}
\RightLabel{w}
\UnaryInfC{$\neg \beta, \alpha, ^\F\!\{ \, \, \}$}
\RightLabel{t}
\UnaryInfC{$^\P\!\{ \neg \beta, \alpha\}$}
\RightLabel{$\P$}
\UnaryInfC{$^\P\!\{ \neg \beta\}, \P\alpha$}
\RightLabel{e}
\UnaryInfC{$\P\alpha, ^\P\!\{ \neg \beta\}$}
\RightLabel{$\H$}
\UnaryInfC{$\P\alpha, \H\neg \beta$}
\end{prooftree}
\begin{prooftree}
\AxiomC{$\begin{array}{c} \vdots \\ \beta, \neg \alpha \end{array}$}
\RightLabel{e}
\UnaryInfC{$\neg \alpha, \beta$}
\RightLabel{w}
\UnaryInfC{$\neg \alpha, \beta, ^\F\!\{ \, \, \}$}
\RightLabel{t}
\UnaryInfC{$^\P\!\{ \neg \alpha, \beta\}$}
\RightLabel{$\P$}
\UnaryInfC{$^\P\!\{ \neg \alpha\}, \P\beta$}
\RightLabel{e}
\UnaryInfC{$\P\beta, ^\P\!\{ \neg \alpha\}$}
\RightLabel{$\H$}
\UnaryInfC{$\P\beta, \H\neg \alpha$}
\end{prooftree}
\end{multicols}
\end{proof}

Now we shall exhibit a construction of the Lindenbaum-Tarski algebra. Consider the quotient set $\Fs\!/\!\equiv$ of the set of formulas $\Fs$ determined by the equivalence relation $\equiv$. We define on $\Fs\!/\!\equiv$ the operations $\wedge$, $\vee$, $\neg$, $\P$, $\F$, $\H$ and $\G$, as usual, in the following way. If $\alpha \in Fm$ we denote by $[\alpha]_{\equiv}$ (or simply $[\alpha]$) the class of equivalence of $\alpha$ determined by $\equiv$. 

\begin{multicols}{2}
\begin{center}
$[\alpha] \wedge [\beta]:=[\alpha \wedge \beta] $,

$[\alpha] \vee [\beta]:=[\alpha \vee \beta] $,

$\neg [\alpha]:=[\neg \alpha] $,

\columnbreak

$\P[\alpha] :=[\P \alpha] $,

$\F[\alpha] :=[\F \alpha] $,

$\H[\alpha] :=[\H \alpha] $,

$\G[\alpha] :=[\G \alpha]$.

\end{center}
\end{multicols}

Since $\equiv$ is, in particular, a congruence relation, we know that the operations are well-defined. Also, if $v_0$ is a propositional variable, we define
$$1=[v_0\vee\neg v_0]$$
and
$$0=[v_0\wedge\neg v_0]$$
It is clear that $\neg 1 = 0$.

\begin{prop}\label{prop5.1} The structure $\langle \Fs\!/\!\equiv , \wedge , \vee , \neg, 0, 1  \rangle$ is a Boolean algebra.
\end{prop}
\begin{proof} This is a long and routine proof. By way of example, we will show two properties. Let $[\alpha] \in \Fs\!/\!\equiv$. Let us prove that $[\alpha]\wedge\neg[\alpha]=0$. That is, $
[\alpha\wedge\neg\alpha]=[\neg(v_0\vee\neg v_0)]$, which, by the definition of the equivalence relation, is equivalent to proving that $
SK_t\vdash\alpha\wedge\neg\alpha,\neg\neg(v_0\vee\neg v_0)
$ and $SK_t\vdash\neg(v_0\vee\neg v_0),\neg(\alpha\wedge\neg\alpha)
$. By Proposition \ref{prop4.1}, it suffices to prove that $
SK_t\vdash\alpha\wedge\neg\alpha,v_0\vee\neg v_0
$ and $
SK_t\vdash\neg(v_0\vee\neg v_0),\neg\alpha\vee\alpha
$. The first follows from the fact that $v_0,\neg v_0$ is an axiom and from the rules $\vee$ and \emph{weakening}. The second follows immediately from Proposition \ref{prop4.2}, the \emph{exchange} rule and, again, the rules $\vee$ and \emph{weakening}. 

Let us now prove that $[\alpha]\wedge 1=[\alpha]$. That is, $
SK_t\vdash\alpha\wedge(v_0\vee\neg v_0),\neg\alpha
$ and $
SK_t\vdash\alpha,\neg\bigl(\alpha\wedge(v_0\vee\neg v_0)\bigr).
$ By Proposition \ref{prop4.1}, this is equivalent to proving that $
SK_t\vdash\alpha\wedge(v_0\vee\neg v_0),\neg\alpha
$ and $
SK_t\vdash\alpha,\neg\alpha\vee\neg(v_0\vee\neg v_0).
$ Taking Proposition \ref{prop4.2} into account, we construct the following proofs:

\begin{multicols}{2}
\begin{prooftree}
\AxiomC{$\alpha, \neg \alpha$}    \RightLabel{e}
\UnaryInfC{$\neg \alpha, \alpha$}
\AxiomC{$v_0,\neg v_0$}    \RightLabel{$\vee$}
\UnaryInfC{$v_0 \vee \neg v_0$}
\RightLabel{w}
\UnaryInfC{$v_0 \vee \neg v_0, \neg \alpha$}
\RightLabel{e}
\UnaryInfC{$\neg \alpha, v_0 \vee \neg v_0$}
\RightLabel{$\wedge$}
\BinaryInfC{$\neg \alpha, \alpha \wedge (v_0 \vee \neg v_0)$}
\RightLabel{e}
\UnaryInfC{$\alpha \wedge (v_0 \vee \neg v_0), \neg \alpha$}
\end{prooftree}

\begin{prooftree}
\AxiomC{$\alpha,\neg \alpha$}
\RightLabel{w}
\UnaryInfC{$\alpha,\neg \alpha, \neg(v_0 \vee \neg v_0)$}
\RightLabel{$\vee$}
\UnaryInfC{$\alpha,\neg \alpha \vee  \neg(v_0 \vee \neg v_0)$}
\end{prooftree}

\end{multicols}

\end{proof}
\noindent We know that the relation $\leq$ defined on $\Fs \!/\!\equiv$ by $[\alpha] \leq [\beta]$ if and only if $[\alpha] \wedge [\beta]=[\alpha]$, is an order relation.
We define on $\Fs \!/\!\equiv$ the relation  $\prec$ in the following way: $[a] \prec [b]$ if and only if  $SK_t \vdash \neg \alpha, \beta$. 

\begin{prop}\label{proporden}
Given $\alpha, \beta \in \Fs$,  $[a]\prec [b]$ if and only if $[\alpha] \leq [\beta]$.   
\end{prop}

\begin{proof}
Let $\alpha, \beta \in \Fs$ such that $[a] \prec [b]$. Then, $SK_t \vdash \neg \alpha, \beta$. Let us see that $[\alpha] \leq [\beta]$, that is, $SK_t \vdash \alpha \wedge  \beta, \neg \alpha$ and $SK_t \vdash \alpha, \neg (\alpha \wedge  \beta)$, which is equivalent to proving $SK_t \vdash \alpha \wedge  \beta, \neg \alpha$ and $SK_t \vdash \alpha, \neg \alpha \vee  \neg \beta$. The latter follows from Proposition \ref{prop4.2} and using the rules \emph{weakening} and $\vee$. The former has the following proof:

\begin{prooftree}
\AxiomC{$\alpha, \neg \alpha$}
\RightLabel{e}
\UnaryInfC{$\neg \alpha, \alpha$}
\AxiomC{$\neg \alpha, \beta$}
\RightLabel{$\wedge$}
\BinaryInfC{$\neg \alpha, \alpha \wedge \beta$}
\RightLabel{e}
\UnaryInfC{$\alpha \wedge \beta, \neg \alpha$}
\end{prooftree}
Suppose that $[\alpha] \leq [\beta]$. Then, $SK_t \vdash \alpha \wedge  \beta, \neg \alpha$ and $SK_t \vdash \alpha, \neg (\alpha \wedge  \beta)$. From $SK_t \vdash \alpha \wedge  \beta, \neg \alpha$ and applying \emph{weakening} and \emph{exchange} we have $SK_t \vdash \neg \alpha, \beta,\alpha \wedge  \beta$. By Proposition \ref{prop4.2}, we have $SK_t \vdash \beta, \neg \beta$ and applying \emph{weakening} and \emph{exchange} we have $SK_t \vdash \neg \alpha, \beta, \neg \alpha, \neg \beta$, and using the rule $\vee$ we get $SK_t \vdash \neg \alpha, \beta, \neg \alpha \vee \neg \beta$ which by definition is $SK_t \vdash \neg \alpha, \beta, \neg (\alpha \wedge \beta)$. Therefore, we obtain $SK_t \vdash \neg \alpha, \beta$ (using the cut rule).
\end{proof}

\begin{prop}\label{prop5.2} The structure $\langle \Fs\!/\!\equiv , \wedge , \vee , \neg, \G ,\H ,\P , \F, 0, 1  \rangle$ is a tense algebra. 
\end{prop}
\begin{proof}
Let us show that $\G 1 =1$. It is enough to show that $\G([v_0\vee\neg v_0]) \equiv [v_0\vee\neg v_0]$, for a propositional variable $v_0$.
That is, we have to show that $\vdash _{SK_t} v_0\vee\neg v_0, \neg \G(v_0\vee\neg v_0)$ and $\vdash _{SK_t} \neg (v_0\vee\neg v_0), \G(v_0\vee\neg v_0)$.
Consider the following proofs.

\begin{prooftree}
\AxiomC{$v_0,\neg v_0$}
\RightLabel{$\vee$}
\UnaryInfC{$v_0\vee\neg v_0$}
\RightLabel{w}
\UnaryInfC{$v_0\vee\neg v_0, \neg \G(v_0\vee\neg v_0)$}
\end{prooftree}
and
\begin{prooftree}
\AxiomC{$v_0,\neg v_0$}
\RightLabel{$\vee$}
\UnaryInfC{$v_0\vee\neg v_0$}
\RightLabel{w}
\UnaryInfC{$v_0\vee\neg v_0, ^\P\!\{ \, \, \}$}
\RightLabel{t}
\UnaryInfC{$^\F\!\{ v_0\vee\neg v_0 \}$}
\RightLabel{$\G$}
\UnaryInfC{$\G(v_0\vee\neg v_0)$}
\RightLabel{w}
\UnaryInfC{$\G(v_0\vee\neg v_0), \neg(v_0\vee\neg v_0)$}
\RightLabel{e}
\UnaryInfC{$\neg(v_0\vee\neg v_0), \G(v_0\vee\neg v_0)$}
\end{prooftree}
Analogously, we prove $\H 1=1$. In order to prove the remaining identities recall that, by Proposition \ref{prop4.2}, for every formula $\alpha$ we have $\vdash_{SK_t}\alpha,\neg \alpha$.\\
Let us prove that $\G([\alpha] \wedge [\beta])=\G[\alpha] \wedge \G[\beta]$, that is, we have to prove that  $\G(\alpha \wedge \beta)\equiv \G \alpha \wedge \G \beta$. Consider the following proofs.
\begin{prooftree}
\AxiomC{$\begin{array}{c} \vdots \\ \alpha, \neg \alpha \end{array}$}
\RightLabel{w}
\UnaryInfC{$\alpha, \neg \alpha, \neg \beta$}
\RightLabel{w}
\UnaryInfC{$\alpha, \neg \alpha, \neg \beta, ^\P\!\{ \, \, \}$}
\RightLabel{t}
\UnaryInfC{$^\F\!\{\alpha, \neg \alpha, \neg \beta\}$}
\RightLabel{$\F$}
\UnaryInfC{$^\F\!\{\alpha, \neg \alpha\}, \F\neg \beta$}
\RightLabel{$\F$}
\UnaryInfC{$^\F\!\{\alpha\}, \F\neg \alpha, \F\neg \beta$}
\RightLabel{e's, t, e}
\UnaryInfC{$^\P\!\{\F\neg \alpha, \F\neg \beta\}, \alpha$}
\AxiomC{$\begin{array}{c} \vdots \\ \beta, \neg \beta \end{array}$}
\RightLabel{w}
\UnaryInfC{$\beta, \neg \beta, \neg \alpha$}
\RightLabel{w}
\UnaryInfC{$\beta, \neg \beta, \neg \alpha, ^\P\!\{ \, \, \}$}
\RightLabel{t}
\UnaryInfC{$^\F\!\{\beta, \neg \beta, \neg \alpha\}$}
\RightLabel{$\F$}
\UnaryInfC{$^\F\!\{\beta, \neg \beta\}, \F\neg \alpha$}
\RightLabel{$\F$}
\UnaryInfC{$^\F\{\beta\}, \F\neg \beta, \F\neg \alpha$}
\RightLabel{e's, t, e}
\UnaryInfC{$^\P\{\F\neg \alpha, \F\neg \beta\}, \beta$}
\RightLabel{$\wedge$}
\BinaryInfC{$^\P\!\{\F\neg \alpha, \F\neg \beta\}, \alpha \wedge \beta$}
\RightLabel{e, t, e}
\UnaryInfC{$\F\neg \alpha, \F\neg \beta, ^\F\!\{\alpha \wedge \beta\}$}
\RightLabel{$\G$}
\UnaryInfC{$\F\neg \alpha, \F\neg \beta, \G(\alpha \wedge \beta)$}
\RightLabel{e's}
\UnaryInfC{$\G(\alpha \wedge \beta), \F\neg \alpha, \F\neg \beta$}
\RightLabel{$\vee$}
\UnaryInfC{$\G(\alpha \wedge \beta), \F\neg \alpha \vee \F\neg \beta$}
\end{prooftree}

\begin{prooftree}
\AxiomC{$\begin{array}{c} \vdots \\ \alpha, \neg \alpha \end{array}$}
\RightLabel{w}
\UnaryInfC{$\alpha, \neg \alpha, \neg \beta$}
\RightLabel{$\vee$}
\UnaryInfC{$\alpha, \neg \alpha \vee \neg \beta$}
\RightLabel{w}
\UnaryInfC{$\alpha, \neg \alpha \vee \neg \beta, ^\P\!\{ \, \, \}$}
\RightLabel{t}
\UnaryInfC{$^\F\!\{\alpha, \neg \alpha \vee \neg \beta\}$}
\RightLabel{$\F$}
\UnaryInfC{$^\F\!\{\alpha\}, \F(\neg \alpha \vee \neg \beta)$}
\RightLabel{e}
\UnaryInfC{$\F(\neg \alpha \vee \neg \beta), ^\F\!\{\alpha\}$}
\RightLabel{$\G$}
\UnaryInfC{$\F(\neg \alpha \vee \neg \beta), \G \alpha$}
\AxiomC{$\begin{array}{c} \vdots \\ \beta, \neg \beta \end{array}$}
\RightLabel{w}
\UnaryInfC{$\beta, \neg \beta, \neg \alpha$}
\RightLabel{e}
\UnaryInfC{$\beta, \neg \alpha, \neg \beta$}
\RightLabel{$\vee$}
\UnaryInfC{$\beta, \neg \alpha \vee \neg \beta$}
\RightLabel{w}
\UnaryInfC{$\beta, \neg \alpha \vee \neg \beta, ^\P\!\{ \, \, \}$}
\RightLabel{t}
\UnaryInfC{$^\F\!\{\beta, \neg \alpha \vee \neg \beta\}$}
\RightLabel{$\F$}
\UnaryInfC{$^\F\!\{\beta\}, \F(\neg \alpha \vee \neg \beta)$}
\RightLabel{e}
\UnaryInfC{$\F(\neg \alpha \vee \neg \beta), ^\F\!\{\beta\}$}
\RightLabel{$\G$}
\UnaryInfC{$\F(\neg \alpha \vee \neg \beta), \G \beta$}
\RightLabel{$\wedge$}
\BinaryInfC{$\F(\neg \alpha \vee \neg \beta), \G \alpha \wedge \G \beta$}
\end{prooftree}

Analogously, we prove that $\H(\alpha\wedge \beta)= \H \alpha \wedge \H \beta$. To prove that $[\alpha] \leq \G\P[\alpha]$ and $[\alpha] \leq \H\F[\alpha]$ we have to prove $[\alpha] \wedge \G\P[\alpha]=[\alpha]$ and $[\alpha] \wedge \H\F[\alpha]=[\alpha]$. Consider the following proofs.
\begin{multicols}{2}
\begin{prooftree}
\AxiomC{$\begin{array}{c} \vdots \\ \alpha, \neg \alpha \end{array}$}
\RightLabel{w}
\UnaryInfC{$\alpha, \neg \alpha, \neg \G\P \alpha$}
\RightLabel{$\vee$}
\UnaryInfC{$\alpha, \neg \alpha \vee \neg \G \P \alpha$}
\RightLabel{e}
\UnaryInfC{$\neg \alpha \vee \neg \G \P \alpha, \alpha$}
\end{prooftree}

\begin{prooftree}
\AxiomC{$\alpha, \neg \alpha$}
\RightLabel{e}
\UnaryInfC{$\neg \alpha, \alpha$}
\AxiomC{$\begin{array}{c} \vdots \\ \alpha, \neg \alpha \end{array}$}
\RightLabel{w}
\UnaryInfC{$\alpha, \neg \alpha, ^\F\!\{ \, \, \}$}
\RightLabel{e}
\UnaryInfC{$\neg \alpha, ^\F\!\{ \, \, \}$}
\RightLabel{t}
\UnaryInfC{$^\P\!\{\neg \alpha, \alpha\}$}
\RightLabel{$\P$}
\UnaryInfC{$^\P\!\{\neg \alpha\}, \P \alpha$}
\RightLabel{e, t, e}
\UnaryInfC{$\neg \alpha, ^\F\!\{ \P \alpha \}$}
\RightLabel{$\G$}
\UnaryInfC{$\neg \alpha, \G\P \alpha$}
\RightLabel{$\wedge$}
\BinaryInfC{$\neg \alpha, \alpha \wedge \G\P \alpha$}
\end{prooftree}    
\end{multicols}

Analogously we prove $[\alpha] \leq \H\F[\alpha]$. Finally, we have to prove that $\P[\alpha]=\neg \H \neg [\alpha]$. This is equivalent to proving  $\P \alpha \equiv \neg \H \neg \alpha$ and by definition of $\neg$, this is equivalent to proving $\P \alpha \equiv \P \neg \neg \alpha$ and this is immediate since $\neg \neg \alpha = \alpha$ (by Proposition \ref{prop4.1}) and $\equiv$ is reflexive. Analogously, we prove $\F[\alpha]=\neg \G \neg [\alpha]$.
\end{proof}

\noindent Consider the function $h:Var \to \Fs\!/\!\equiv$, where $Var$ is the set of all propositional variables, defined in the following way, for every $v\in Var$
$$h(v)=[v]$$
Then, $h$ can be extended to $\overline{h}:\Fs \to \Fs\!/\!\equiv$ in a natural way as follows

\begin{multicols}{2}
\begin{enumerate}[]
  \item $\overline{h}(v)=h(v)=[v]$, $v\in Var$,
  \item $\overline{h}(\alpha \wedge \beta)=\overline{h}(\alpha) \wedge \overline{h}(\beta)$,
  \item $\overline{h}(\alpha \vee \beta)=\overline{h}(\alpha) \vee \overline{h}(\beta)$, 
  \item $\overline{h}(\neg \alpha)=\neg \overline{h}(\alpha)$,
  \item $\overline{h}(\P \alpha)=\P \overline{h}(\alpha)$,
  \item $\overline{h}(\F \alpha)=\F \overline{h}(\alpha)$,
  \item $\overline{h}(\H \alpha)=\H \overline{h}(\alpha)$,
  \item $\overline{h}(\G \alpha)=\G \overline{h}(\alpha)$.
\end{enumerate}
\end{multicols}
It is clear that $\overline{h}$ is a homomorphism.
\begin{lem} For every $\alpha\in \Fs$, $\overline{h}(\alpha)=[\alpha]$.
\end{lem}
\begin{proof} By induction on the complexity of $\alpha$.
\end{proof}

\begin{theo} (Completeness for formulas) If $\alpha\in Fm$ is valid then $SK_t \vdash\alpha$.
\end{theo}
\begin{proof}
Suppose that $\alpha$ is valid, then $f(\alpha)=1$ for every homomorphism $f:Fm \to A$ for every  $A \in \mathbb{B}_t$. In particular, $\overline{h}(\alpha)=1$. Then, $\overline{\alpha}=1$ which means that  $\alpha \equiv v_0\vee\neg v_0$, for some propositional variable $v_0$. By definition of $\equiv$, we have that $SK_t \vdash \alpha, \neg (v_0\vee\neg v_0)$ y $SK_t \vdash \neg \alpha, v_0\vee\neg v_0$. Then, we can construct the following proof

\begin{prooftree}
\AxiomC{$v_0, \neg v_0$}
\RightLabel{$\vee$}
\UnaryInfC{$v_0\vee \neg v_0$}
\RightLabel{w}
\UnaryInfC{$v_0\vee \neg v_0 ,\alpha$}
\RightLabel{e}
\UnaryInfC{$\alpha, v_0\vee \neg v_0$}
\AxiomC{$\begin{array}{c} \vdots \\ \alpha, \neg(v_0\vee \neg v_0) \end{array}$}
\RightLabel{cut}
\BinaryInfC{$\alpha$}
\end{prooftree}
Therefore, $SK_t \vdash \alpha$. \end{proof}

\section{Conclusions}

In this paper, we have studied the minimal tense logic $K_t$ from both algebraic and proof-theoretic perspectives. We first introduced the degree-of-truth-preserving logic associated with the class of tense Boolean algebras and proved that it coincides with $K_t$. This provides an alternative characterization of the minimal tense logic in terms of preservation of degrees of truth and shows that the standard algebraic semantics of tense Boolean algebras can be naturally connected with a degree-preserving consequence relation.

We then considered Kashima's Gentzen-style calculus for $K_t$ and established its soundness and completeness with respect to the degree-of-truth-preserving semantics of tense Boolean algebras. The main ingredient in the completeness proof was an adaptation of the Lindenbaum--Tarski construction to Kashima's nested sequent framework. This construction makes it possible to construct the relevant tense Boolean algebra directly from the proof-theoretic system, thereby providing a direct algebraic completeness proof for Kashima's calculus.

The connection established here between Kashima's nested-sequent calculus and tense Boolean algebras suggests several directions for further research. In particular, it would be interesting to investigate whether the algebraic methods for cut elimination developed by Belardinelli, Jipsen, and Ono \cite{BJO} can be extended to the present tense setting. Such an extension could provide an algebraic counterpart of the cut-elimination theorem for Kashima's calculus and, more generally, shed light on the relationship between nested-sequent proof theory and the algebraic structure of tense Boolean algebras.

\end{document}